\documentclass{article}
\usepackage{geometry}
\usepackage[utf8]{inputenc}
\usepackage[english]{babel}
\usepackage{amssymb, amsthm, amsmath}
\usepackage{amsfonts}
\usepackage{booktabs}
\usepackage{cases}

\usepackage{comment}
\usepackage{xcolor}
\usepackage{amssymb}
\usepackage{scrextend}
\usepackage{graphicx}
\usepackage{tikz}
\usepackage{hyperref}
\usepackage[capitalise, noabbrev, nameinlink]{cleveref}
\usepackage{subcaption}

\newcommand{\G}{\mathcal{G}}
\DeclareMathOperator{\tr}{tr}

\newtheorem{thm}{Theorem}[section]
\newtheorem{cor}[thm]{Corollary}
\newtheorem{lem}[thm]{Lemma}
\newtheorem{prop}[thm]{Proposition}

\newtheorem{obs}[thm]{Observation}

\newtheorem{defn}[thm]{Definition}
\newtheorem{ex}[thm]{Example}

\newtheorem*{claim*}{Claim}
\newtheorem{case}{Case}

\newcommand{\abs}[1]{\left\vert#1\right\vert}

\title{
The Strong Spectral Property of Graphs: Graph Operations and Barbell Partitions}
\author{Sarah Allred \thanks{Department of Mathematics, Vanderbilt University, Nashville, TN, USA (sarah.e.allred@vanderbilt.edu)}
\and Emelie Curl
\thanks{Department of Mathematics, Statistics, and Computer Science, Hollins University, Roanoke, VA, USA (curlej@hollins.edu)}
\and Shaun Fallat
\thanks{Department of Mathematics and Statistics, University of Regina, Regina, Saskatchewan, CA (shaun.fallat@uregina.ca)}
\and Shahla Nasserasr \thanks{School of Mathematical Sciences, Rochester Institute of Technology, Rochester, NY (sxnsma@rit.edu)}
\and Houston Schuerger \thanks{Department of Mathematics, Trinity College, Hartford, CT, USA (houston.schuerger@trincoll.edu)}
\and
Ralihe R. Villagr\'{a}n \thanks{Department of Mathematics and Computer Science, Eindhoven University of Technology, The Netherlands (r.r.villagran.olivas@tue.nl)}
\and Prateek K. Vishwakarma \thanks{Department of Mathematics and Statistics, University of Regina, Regina, Saskatchewan, CA (prateek.vishwakarma@uregina.ca)}
}
\date{\today}
\begin{document}
\maketitle
\begin{abstract}
The utility of a matrix satisfying the Strong Spectral Property has been well established particularly in connection with the inverse eigenvalue problem for graphs. More recently the class of graphs in which all associated symmetric matrices possess the Strong Spectral Property (denoted $G^{SSP}$) were studied, and along these lines we aim to study properties of graphs that exhibit a so-called barbell partition. Such a partition is a known impediment to membership in the class $G^{SSP}$. In particular we consider the existence of barbell partitions under various standard and useful graph operations.

\end{abstract}



\section{Introduction}

The {\em inverse eigenvalue problem of a graph (IEPG)} refers to determining all possible spectra of real symmetric matrices whose pattern of nonzero off-diagonal entries is described by the edges of a given graph (see \cite{IEPG2, FH, Fal, HLA2, HLS2022inverse}).  

Studies on the IEPG have focused on topics such as: determining the maximum eigenvalue multiplicity, or equivalently, maximum nullity, or equivalently again, minimum rank of matrices described by the graph.  Computing the maximum multiplicity in general remains unresolved and an active area of research (see \cite{FH, HLA2} for  extensive bibliographies).  More recently, there has been progress on the related question of determining the minimum number of distinct eigenvalues of matrices  described by a given graph \cite{AACF, BFH}.  

Maximum nullity, minimum number of distinct eigenvalues, and other related notions help to resolve instances of the inverse eigenvalue problem for a specific graph or family of graphs, but a general solution is far from known. Recently, new developments building upon a known matrix property called the Strong Arnold Property (see \cite{Arnold71, CdVF, CdV}), known as the Strong Spectral Property (SSP) and the Strong Multiplicity Property (SMP), have been used in connection with the IEPG \cite{BFH,IEPG2} and seem to be promising tools for working on the IEPG in more general terms.

All matrices are real and symmetric; $O$ and $I$ denote  zero and identity matrices of appropriate size, respectively. 
%For an $n\times n$ matrix $M$ and $\alpha,\beta \subseteq \{1,2,\ldots,n\}$,
%the submatrix of $M$ lying in rows indexed by
%$\alpha$ and columns indexed by $\beta$ is denoted
%by $M[\alpha,\beta]$; in the case that $\beta=\{1,2,\ldots,n\}$ this can be denoted by $M[\alpha, :]$, and similarly for $\alpha=\{1,2,\ldots,n\}$.
A symmetric matrix $A$ has the {\it Strong Arnold Property} (or $A$ has the SAP for short)
if the only symmetric matrix $X$ satisfying $A\circ X=O$, $I\circ X=O$ and $AX=O$ is $X=O$ (recall that product $\circ$ is the entry-wise matrix product).  
An $n \times n$ symmetric matrix $A$ satisfies the {\it Strong Multiplicity Property} (or $A$ has the SMP)
provided the only symmetric matrix $X$ satisfying $A\circ X=O$, $I\circ X=O$, $[A,X]=O$, and $\tr(A^iX)=0$ for $i=2,\ldots, n-1$ is $X=O$ \cite[Definition~18 and Remark~19]{BFH}. 
A symmetric matrix $A$ has the {\it Strong Spectral Property} (or $A$ has the SSP)
if the only symmetric matrix $X$ satisfying $A\circ X=O$, $I\circ X=O$ and $[A,X]=O$ is $X=O$ \cite[Definition~8]{BFH}.  It follows from the definitions above that the SSP implies the SMP, and the SMP implies $A+\lambda I$ has the SAP for every real number $\lambda$ (see \cite{BFH,IEPG2}).

The {\em graph} $\G(A)$ of a real symmetric $n\times n$ matrix $A=[a_{ij}]$  is the (simple, undirected, finite) graph with vertices $\{1,\dots,n\}$ and edges $ij$ such that $i\ne j$ and $a_{ij} \neq 0$.   For a  graph $G=(V,E)$ with  vertex set $V=\{1,\dots,n\}$ and edge set $E$, 
the {\em set of symmetric matrices described by $G$}, $\mathcal{S}(G)$,  is the set of all real symmetric $n \times n$ matrices $A=[a_{ij}]$  such that  $\G(A)=G$.  The IEPG for $G$ asks for the determination of all possible spectra of matrices in $\mathcal{S}(G)$. The number of distinct eigenvalues of $A$ is denoted by $q(A)$, and   $q(G) = \min\{q(A)\; \mbox{:}\; A \in \mathcal{S}(G)\}$. 
%For any square matrix $A$, we let $\sigma(A)$ denote the multiset consisting of the  spectrum of $A$. \textcolor{blue}{I don't think we use $\sigma(A)$ throughout the paper.}

Given a graph $G=(V,E)$, and $S\subset V$, we let $N_{G}(S)$ denote the set of all vertices adjacent to some vertex in $S$. In particular, if $v$ is a vertex of a graph $G$, the {\em neighborhood} of $v$ is the set of vertices adjacent to $v$, and is denoted by $N_G(v)$. Further, the degree of $v$, denoted by $\deg(v)$, is equal to $|N_G(v)|$. If $S \subseteq V$, then we let $G[S]$ denote the subgraph of $G$ induced by the vertices in $G$ (sometimes referred to as the vertex induced subgraph of $G$). An edge $e$ in $G$ with end points $u$ and $v$ is denoted by $e=\{u,v\}$ of for brevity we may just write $uv\in E(G)$. We let $G \pm e$ denote the graph obtained from $G$ by adding a new edge $e$ or by removing the existing edge $e$ from $G$. Similarly, if $S\subseteq V$, we let $G-S$ denote the induced subgraph of $G$ obtained by removing $S$ from $V$. 
%\textcolor{blue}{Sometimes $\setminus$ is used}. 
If $S=\{v\}$, we use abbreviation $G-v$ to denote the graph obtained from $G$ by removing the vertex $v$. If $G$ and $H$ are two graphs, then the join of $G$ and $H$, denoted by $G \vee H$, is the graph obtained from the disjoint union of $G$ and $H$ and adding all possible edges between the vertices in $G$ and the vertices in $H$. Finally, as is standard, we let $K_n$ and $P_n$, $n\geq 1$, and $C_n$, $n\geq 3$, denote the complete graph, the path graph, and the cycle on $n$ vertices, respectively. 

Suppose $G$ is a given graph and $v$ is a fixed vertex in $G$. The notion of duplicating (or cloning) a vertex is a natural graph operation and has interesting implications on the inverse eigenvalue problem for graphs (see \cite{AAB, LOS}). There are two versions of duplicating a vertex, namely with an edge or without an edge. The graph $jdup(G,v)$ ($dup(G,v)$) is the graph obtained from $G$ by adding a new vertex $u$ and connecting $u$ to all the vertices in $N_{G}(v)$ and $v$ (connecting $u$ to all the vertices in $N_{G}(v)$).

The main focus of this work is to study when the SSP is preserved under certain standard graph operations. One inherent flaw with this approach is that the SSP is a matrix property and not necessarily a graph property. As such, we are very interested in the graphs $G$ for which the SSP holds for all  $A \in \mathcal{S}(G)$. This class is denoted by $G^{SSP}$ and was introduced in \cite{SSPforGs}. In \cite{SSPforGs}, the concept of a barbell partition was noted and a connection to the complement of the class $G^{SSP}$ was established. In this paper, we work with barbell partitions extensively and work out numerous relationships regarding barbell partitions and standard graph operations. Such analysis leads to a better understanding of graphs that do not belong to the set $G^{SSP}$. In section 2 of this work we recall the definition of barbell partitions and verify classes of graphs that exhibit barbell partitions. In the remaining sections (sections 3-5) we discuss the preservation of barbell partitions under various graph operations and graph products.

Concerning graphs $G$ and the class $G^{SSP}$, it is known that the cycle on 4 or more vertices does not belong to the class $G^{SSP}$. On the other hand it is also known (see \cite{FHLS}) that for the 4-cycle any realizable multiplicity list can be realized by a matrix with the SSP. However, for any $n$-cycle it is still not clear if every realizable multiplicity list can be realized by a matrix with the SSP. As a matter of pushing the narrative a bit further we end the introduction with an example of a class of matrices in $S(C_n)$ with $n$ even that possess the SSP.

\begin{prop}
For $A \in \mathcal{S}(C_n)$, if $n$ is even and $0 \neq \lambda = a_{11} = \cdots = a_{\frac{n}{2}\frac{n}{2}}$, $0 \neq -\lambda = a_{\frac{n}{2}+1 \frac{n}{2}+1} = \cdots = a_{nn}$, and all other entries in $A$ are a common value $b \neq 0$, then $A$ has the SSP. 
\end{prop}
\begin{proof}
Suppose $A \in S(C_n)$ such that $n$ is even and 
\[A = \begin{bmatrix}
\lambda & b & 0 & 0 & 0 & \cdots & 0 & \cdots & 0 & b\\
b & \lambda & b & 0 & 0 & \cdots & 0 & \cdots & 0 & 0\\
0 & b & \lambda & b & 0 & \cdots & 0 & \cdots & 0 & 0\\
\vdots & \ddots & \ddots & \ddots & \ddots & \cdots & \vdots & \cdots & \vdots & \vdots\\
0 & 0 & \cdots & b &  \lambda & b & 0 & \cdots & 0 & 0\\
0 & 0 & \cdots & 0 & b & -\lambda & b & \cdots & 0 & 0\\
\vdots & \vdots & \cdots & \vdots & \ddots & \ddots & \ddots & \ddots & \vdots & \vdots\\
0 & \cdots & \cdots & \cdots  & \cdots &  0 & b & -\lambda & b & 0\\
0 & \cdots & \cdots & \cdots & \cdots &  \cdots  &  0 & b & -\lambda & b\\
b & \cdots & \cdots & \cdots & \cdots & \cdots & \cdots & 0 & b & -\lambda\\
\end{bmatrix}\]
such that $\lambda$ and $b$ are real numbers such that $\lambda, b \neq 0$.  Consider an $n \times n$ real symmetric matrix $X=[x_{ij}]$ such that: (1) $A \circ X = O$, (2) $X \circ I = O$, and (3) $AX = XA$. Then $XA = AX$, yields the equations $bx_{13} = \cdots = bx_{kk+2} = bx_{k+1 k -1} = \cdots = bx_{n2}$ along the super and sub diagonals and in the entry in the $(1,n)$ and $(n,1)$ positions. Since $b \neq 0$, this yields that 
$x_{13} = \cdots = x_{k k+2} = x_{k+1 k - 1} = \cdots = x_{n2}$. Considering the next two diagonals and entries in positions $(1,n-1)$, $(2,n)$, $(n-1,1)$, and $(n,2)$ yields the equations $0 = x_{13} = \ldots = x_{k k+2} = x_{k+1 k -1} = \ldots = x_{n2}$. Continuing the argument in this manner will show that $X=0$ and hence $A$ has the SSP.
\end{proof}

%\begin{conj}
%For $A \in \mathcal{S}(C_n)$ in order for $A$ to possess the SSP, one of the following conditions must be met
%\begin{itemize}
%    \item[1.] If $n$ is even, $0 \neq \lambda = a_{11} = \cdots = a_{\frac{n}{2}\frac{n}{2}}$, $0 \neq -\lambda = a_{\frac{n}{2}+ 1\frac{n}{2} + 1} = \cdots = a_{nn}$ (half the diagonal entries are the same and the other half have opposite sign) and $|b_{34}| = |b_{56}| = |b_{89}| = |b_{i+2 i+2}| \cdots \neq |b_{12}| = |b_{45}| = |b_{78}| = |b_{i+3 i+3}|$ (the super or sub diagonal entries alternate in absolute value) 
 %   \item[2.]If $n$ is odd, .... ?  
%\end{itemize}
%\end{conj}

\section{Barbell Partitions}

We now turn our attention to the interesting class of graphs that do not  belong to $G^{SSP}$. 
%By definition any graph operation closed within this class necessarily will correspond to a graph operation that preserves the SSP. 
In connection with the class of graphs that do not belong to $G^{SSP}$ (see \cite{SSPforGs}), we consider the notion of a barbell partition of the vertex set. For our purposes, given a set $U$ we will consider a partition of $U$ to be a collection of pairwise disjoint subsets of $U$ such that their union is $U$. The subsets can be empty. The concept of barbell partition is defined in \cite{SSPforGs}. 

\begin{defn}
A barbell partition of a graph $G$ is a partition of $V(G)$ into three disjoint parts $\{R,W_1,W_2\}$ such that:
\begin{enumerate}
    \item $R$ is allowed to be an empty set, but $W_i \neq \emptyset$ for $i \in \{1,2\}$
    \item there are no edges between vertices in $W_1$ and vertices in $W_2$
    \item for each $v \in R$, $\left\vert N_G(v) \cap W_i \right\vert \neq 1$ for $i \in \{1,2\}$.
\end{enumerate}
\end{defn}

\noindent It was then shown that if a graph $G$ has a barbell partition, then $G \not \in  G^{SSP}$.

\begin{lem} \cite{SSPforGs}
Let G be a graph with a barbell partition. Then
there is a matrix $M \in \mathcal{S}(G)$ such that $M$ does not have the SAP (and the SSP).
\end{lem}

It thus follows that any graph with a barbell partition is not a member of $G^{SSP}$, $G^{SAP}$, or $G^{SMP}$.  Our results concerning barbell partitions will often conclude that a graph is not a member of $G^{SSP}$.  However, in each such case, it can also be said that the graph in question is neither a member of $G^{SAP}$ nor of $G^{SMP}$. We observe, for completeness, that the converse to the previous lemma need not hold in general. Consider the cycle on 4 vertices with two additional pendant vertices adjacent to non-adjacent vertices on this 4-cycle. This graph is not in $G^{SSP}$ and does not have a barbell partition (see also \cite{SSPforGs}). 

As a result, it becomes natural to discuss barbell partitions when considering SSP.  To this end, the remaining sections will focus on classes of graphs which have barbell partitions and graph operations which preserve the presence of barbell partitions or in some cases introduce barbell partitions.  Furthermore, in \cite{fort} the concept of a fort was introduced and defined to be a subset of a graph's vertices such that no vertex not in the fort is adjacent to exactly one vertex in the fort.  For the purposes of this paper, we will follow the convention that every nontrivial graph $G$ has a fort, specifically $V(G)$, with the criterion in the definition being understood to be vacuously true in this case.  The following theorem establishing the connection between forts and zero forcing (see \cite{HLA2} for more details about zero forcing) follows from \cite[Thm. 3]{fort}.

\begin{thm}\label{fort}\cite{fort}
Let $G$ be a graph and $S \subseteq V(G)$.  Then $V(G) - S$ is a zero forcing set of $G$ if and only if $S$ does not contain a fort.
\end{thm}

The definition and results concerning forts are worth noting because while the original definition of barbell partitions does not mention forts, one can define barbell partitions utilizing the concept of forts thus creating a connection between the two areas of study.  With this in mind, we first introduce the concept of a pair of separated forts and then identify that this concept can provide an equivalent definition for barbell partitions.

\begin{defn}
Let $G$ be a graph.  If $W_1$ and $W_2$ are disjoint nonempty forts in $G$ and no vertex in $W_1$ is adjacent to a vertex in $W_2$, then we say that $\{W_1,W_2\}$ is a pair of separated forts.
\end{defn}

\begin{obs}\label{fortbarb}
Let $G$ be a graph.  Then $\{V(G) \setminus (W_1 \cup W_2),W_1,W_2\}$ is a barbell partition of $G$ if and only if $\{W_1,W_2\}$ is a pair of separated forts of $G$.
\end{obs}

We begin our discussion of barbell partitions by providing a list of small observations which will prove useful later on.

\begin{obs}\label{discon}
If a graph is disconnected, then there is a barbell partition $W_1=H_1$, and $W_2=V(G)\setminus H_1$, and $R=\emptyset.$
\end{obs}

\begin{obs}
If $G$ is a graph such that
\begin{itemize}
    \item $S$ is a cut-set of $G$
    \item $\mathcal H$ is the collection of components of $G-S$ and
    \item for $H \in \mathcal H$,
    \[v \in N_G\left(V(H)\right) \cap S \Longrightarrow \left\vert N_G(v) \cap V(H) \right\vert \geq 2,\]
\end{itemize}
then $G$ has a barbell partition with $R=S$, $W_1=V(H')$ for some $H' \in \mathcal H$, and $W_2=\left(\bigcup_{H \in \mathcal H}V(H)\right)\setminus V(H')$, and thus must not be in $ G^{SSP}$.
\end{obs}

In \cite{SSPforGs} the following results concerning barbell partitions, trees, and unicyclic graphs were proven.

\begin{cor}\cite{SSPforGs}
If a tree $T$ has a vertex $v$ with $\deg(v) \geq 4$ or two vertices $u,v$ with $\deg(u),\deg(v) \geq 3$, then $T$ admits a barbell partition. Hence $T \not \in  G^{SSP}$ .
\end{cor}

%\textcolor{blue}{statements like "(admits a barbell partition)" or $G \not \in  G^{SSP}$ appear in the statements of this section but not later. Should we delete the ones here?}

\begin{cor}\cite{SSPforGs}
If a unicyclic graph $G$ has a vertex $v$ such that $\deg(v) \geq 4$, or a vertex $u$ not contained in the cycle with $\deg(u) \geq 3$, then $G$ admists a barbell partition. Hence $G \not \in G^{SSP}$ .
\end{cor}

With these results in mind, we now present a result concerning cacti, where a cactus is defined to be a connected graph such that no edge is contained in more than one cycle.

\begin{thm}
Let $G$ be a cactus with at least two cycles.  Then $G$ has a barbell partition, and thus $G \not \in  G^{SSP}$.
\end{thm}

\begin{proof}
We will prove the theorem by breaking into two cases, and then showing that in either case $G$ has a barbell partition and thus $G \not \in  G^{SSP}$.
\vspace{0.1in}

\noindent \textbf{Case 1:} There exists a vertex $v \in V(G)$ such that $v$ is in two cycles $H_1$ and $H_2$. 
\vspace{0.1in}\\
%\begin{figure}[h]
  %\centering
  %\includegraphics[width=0.9\linewidth]{cactus_barbell_1.png}
  %\caption{A example of a cactus $G$ with a vertex $v$ in two cycles}
  %\label{cactus_barbell_1}
%\end{figure}
Since $G$ is a cactus, no edge in $E(G)$ is in more than one cycle.  As a result, every vertex of degree larger than two is a cut-vertex, and so in particular, $v$ is a cut-vertex.  Let $u_{1,1},u_{1,2},u_{2,1},u_{2,2} \in N_G(v)$ such that $u_{1,1}, u_{1,2} \in V(H_1-v)$ and $u_{2,1}, u_{2,2} \in V(H_2-v)$, and let $H$ be the component of $G-v$ containing $H_1 - v$.  Finally, let $R=\{v\}$, $W_1=V(H)$, and $W_2=V(G)\setminus \left(V(H)\cup\{v\} \right)$.
\vspace{0.1in}
\begin{addmargin}[0.87cm]{0cm}{\bf Claim 1:}
$W_1,W_2 \not = \emptyset$ and there do not exist vertices $w_1 \in W_1$ and $w_2 \in W_2$ such that $w_1w_2 \in E(G)$.
\vspace{0.1in}

\noindent {\em Proof of Claim 1.}  Since $u_{1,1},u_{1,2} \in W_1$ and $u_{2,1},u_{2,2} \in W_2$, it follows that $W_i \neq \emptyset$ for $i \in \{1,2\}$.  Furthermore, since $v$ is a cut-vertex of $G$ and $H$ is a component of $G-v$, it follows that there are no vertices $w_1 \in W_1=V(H)$ and $w_2= W_2\in V(G) \setminus \left(V(H) \cup \{v\}\right)$ such that $w_1w_2 \in E(G)$.
\end{addmargin}
\vspace{0.1in}

\begin{addmargin}[0.87cm]{0cm}{\bf Claim 2:}
$\left\vert N_G(v) \cap W_i \right\vert \neq 1$ for $i \in \{1,2\}$.
\vspace{0.1in}

\noindent {\em Proof of Claim 2.} 
\vspace{0.1in}\\
Since $u_{1,1},u_{1,2} \in W_1$ and $u_{2,1},u_{2,2} \in W_2$, it follows that $\left\vert N_G(v) \cap W_i \right\vert \geq 2$ for $i \in \{1,2\}$.
\end{addmargin}
\vspace{0.1in}

\noindent \textbf{Case 2:} There does not exist a vertex $v \in V(G)$ such that $v$ is in two cycles.
%\begin{figure}[h]
  %\centering
  %\includegraphics[width=0.9\linewidth]{cactus_barbell_2.png}
  %\caption{Example of a cactus $G$ with vertices $v_1$ and $v_2$ in two disjoint cycles}
  %\label{cactus_barbell_2}
%\end{figure}
\vspace{0.1in}\\
Since $G$ is a cactus containing at least two cycles, there exist vertices $v_1,v_2 \in V(G)$ such that $v_1$ is contained in a cycle $H_1$, $v_2$ is contained in a cycle $H_2$, and the $v_1v_2$-path $P$ in $G$ contains no vertices which are incident to a cycle except for $v_1$ and $v_2$. Note that the path $P$ could be could be an edge. As before, let $u_{1,1},u_{1,2} \in N_G(v_1)$ and $u_{2,1},u_{2,2} \in N_G(v_2)$ such that $u_{1,1}, u_{1,2} \in V(H_1-v_1)$ and $u_{2,1}, u_{2,2} \in V(H_2-v_2)$.  Since $v_1$ and $v_2$ are vertices in a cactus each contained in a cycle $H_i$ and a path $P$ such that $V(P-v_i) \cap V(H_i)=\emptyset$, it follows that each of $v_1$ and $v_2$ are cut-vertices of $G$.  If $V(P)\setminus \{v_1,v_2\}=\emptyset$, then let $R=\{v_1,v_2\}$; and otherwise let $R=V(H_P)\cup\{v_1,v_2\}$ for $H_P$ the component of $G- \{v_1,v_2\}$ containing $P-\{v_1,v_2\}$.  Additionally, let $W_1=V\left(\bigcup \mathcal H \right)$ for $\mathcal H$ the collection of components of $G-(\{v_1\}\cup R)$ not containing $v_2$ and $W_2=V(G)\setminus \left(R \cup W_1 \right)$.  
\vspace{0.1in}
\begin{addmargin}[0.87cm]{0cm}{\bf Claim 1:}
$W_1,W_2 \not = \emptyset$ and there do not exist vertices $w_1 \in W_1$ and $w_2 \in W_2$ such that $w_1w_2 \in E(G)$.
\vspace{0.1in}

\noindent {\em Proof of Claim 1.}
Since $v_1$ and $v_2$ are cut vertices of $G$, we have $W_1,W_2 \not = \emptyset$.  Furthermore, since $W_1=V(\bigcup \mathcal H)$ and $\mathcal H$ is a collection of components of $G-v_1$, it follows that the only vertex not in $W_1$ adjacent to vertices in $W_1$ is $v_1$.  Since $v_1 \not \in W_2$, there do not exist vertices $w_1 \in W_1$ and $w_2 \in W_2$ such that $w_1w_2 \in E(G)$.
\end{addmargin}

\vspace{0.1in}

\begin{addmargin}[0.87cm]{0cm}{\bf Claim 2:}
For $v \in R$ and $i \in \{1,2\}$, $\left\vert N_G(v) \cap W_i \right\vert \neq 1$.
\vspace{0.1in}

\noindent {\em Proof of Claim 2.} 
Again, the only vertex not in $W_1$ adjacent to vertices in $W_1$ is $v_1$.  Likewise, the only vertex not in $W_2$ adjacent to vertices in $W_2$ is $v_2$.  Furthermore, since $u_{1,1},u_{1,2} \in W_1$ and $u_{2,1},u_{2,2} \in W_2$, it follows that $\left\vert N_G(v_i) \cap W_i \right\vert \geq 2$ for $i \in \{1,2\}$, and thus for $v \in R$ and $i \in \{1,2\}$, $\left\vert N_G(v) \cap W_i \right\vert \neq 1$.
\end{addmargin}
\vspace{0.1in}
\end{proof}

%\textcolor{blue}{(from Ralihe) I think that 1.39 can also be proven has follows:  First, suppose $P$ admits a barbell partition $(R, W_1,W_2)$. Then there exists $v_i\in W_i$ and $r\in R$ such that $r$ is adjacent to $v_1$ and to $v_2$ since the diameter of the Petersen graph is $2$. Moreover, we have that $P$ is $3$-regular. Therefore $|N(r)\cap W_i|=1$ for some $i=1,2$. Concluding that $P$ does not admit a barbell partition.}\\

In an effort to develop a catalog of graphs that are excluded from the class $G^{SSP}$, we are interested in graphs that admit a barbell partition. On the other hand, we have the following graph structural result regarding the nonexistence of a barbell partition. 

\begin{lem}\label{degdiam}
Let $G$ be a connected graph with diameter $2$ and maximum degree $3$. Then $G$ does not admit a barbell partition.
\end{lem}

\begin{proof}
Let $W_1,W_2 \subseteq V(G)$ be arbitrary nonempty sets of vertices such that no vertex in $W_1$ is adjacent to a vertex in $W_2$.  Let $R=V(G) \setminus (W_1 \cup W_2)$.  Since $G$ is connected, $R$ is nonempty. We will show that there exists a vertex $r \in R$ such that for some $i \in \{1,2\}$, $\big{|}N_G(r) \cap W_i\big{|}=1$. 

Let $w_1 \in W_1$ and $w_2 \in W_2$.  Since $G$ is connected and of diameter 2 there exists a $w_1w_2$-path of length 2. However, since there are no edges between vertices in $W_1$ and vertices in $W_2$ and this path is of length 2, the middle vertex in this $w_1w_2$-path must be a member of $R$, call it $r$.  Since $r$ is neighbors with $w_1$ and $w_2$, but $G$ has maximum degree $3$, $r$ has at most one other neighbor.  Thus, for some $i \in \{1,2\}$, $|N_G(r)\cap W_i|=1$.
\end{proof}

\begin{cor}
The Petersen graph, $P,$ does not admit a barbell partition.
\end{cor}

\begin{proof}
$P$ is $3$-regular and has diameter 2, so by Lemma \ref{degdiam}, $P$ does not admit a barbell partition. %{\color{red}(Could we also keep the alternate proof which was written specifically for the Petersen graph?)}
\end{proof}

%\begin{proof}
%Suppose $P$ admits the barbel partition $(R,W_1,W_2).$ Without loss of generality, assume that each vertex in $R$ has at least one neighbour each, outside and inside $R.$ Since $P$ is $3$-regular and $(R,W_1,W_2)$ is a barbel partition, each vertex in $R$ must have two neighbours in exactly one of the $W_i$'s and remaining one in $R$ itself. This means that the set of neighbours of $R$ in $W_i,$ $N_{W_i}(R)$ is even (and nonempty) for $i=1,2.$ Also note that $N_{W_1}(R)\cup N_{W_2}(R)$ is an independent set because $P$ does not contain $3$-cycles. Now we use the fact that the independence number, size of the largest independent vertex set for $P$ is $4$ [[cite: \textit{The Petersen Graph} by Holton and Sheehan, Cambridge University Press (1993), pp 12]]. This means that each $N_{W_i}(R)=2,$ which in turn implies that the cardinality of $R$ is $2.$ This is a contradiction because the vertex connectivity of $P$ is $3.$ [[cite: \textit{The Petersen Graph} by Holton and Sheehan, Cambridge University Press (1993), pp 14]].
%\end{proof}

It is necessary that the graph in Lemma \ref{degdiam} has all three properties, that is be connected, maximum degree $3$, and have diameter 2.  As stated in a previous theorem, every disconnected graph admits a barbell partition, and in the diagram below we provide an example of a graph of diameter 2 and a 3-regular graph which each admit barbell partitions.

\begin{figure}[h]
  \centering
  \includegraphics[width=0.7\linewidth]{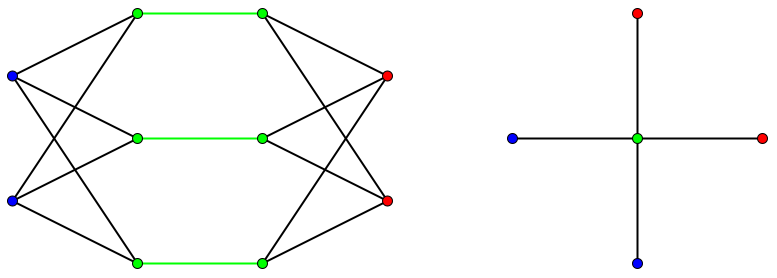}
  \caption{A 3-regular graph and a graph of diameter 2 each of which are connected but admit a barbell partition. The colors green, red, and blue represent $R, W_1$, and  $W_2$, respectively.}
\end{figure}

\section{Barbell Partitions and the Removal or Addition of Edges or Vertices}

Having identified some classes of graphs admitting barbell partitions, we now consider the effect basic graph operations such as adding or removing vertices or edges have on the presence of barbell partitions.  We first consider the effect of adding or removing an edge.

\begin{obs}
Let $G$ be a graph and $R,W_1,W_2$ be a partition of the vertices of $G$ with $W_1, W_2 \neq \emptyset$ such that no vertex in $W_1$  is adjacent to a vertex in $W_2$.  Let $e=\{u,v\} \in E(G)$ with $u,v \in W_i$ for some $i \in \{1,2\}$ or $u,v \in R$.  Then $\{R,W_1,W_2\}$ is a barbell partition of $G$ if and only if $\{R,W_1,W_2\}$ is a barbell partition of $G-e$.  
\end{obs}

Since the pair of vertices $u$ and $v$ are in the same element of the partition, the presence or lack of the edge $e=\{u,v\}$ has no effect on any of the three criteria determining whether $\{R,W_1,W_2\}$ is a barbell partition or not.

\begin{obs}
Let $G$ be a graph admitting a barbell partition $\{R,W_1,W_2\}$, let $u \in R$, let $v \in W_{i}$ for some $i \in \{1,2\}$, and let $e=\{u,v\} \in E(G)$.  Then $G-e$ admits a barbell partition provided $\left\vert N_G(u) \cap W_{i} \right\vert>2$.
\end{obs}

In this second case, the removal of $e$ only effects the number of neighbors $u \in R$ has in $W_{i}$, and since $\left\vert N_G(u) \cap W_{i} \right\vert>2$, $\{R, W_1, W_2\}$ is still a barbell partition of $G-e$.  The effect of adding an edge is quite similar and we have the following.

\begin{obs}
Let $G$ be a graph admitting a barbell partition $\{R,W_1,W_2\}$, let $u \in R$, let $v \in W_{i}$ for $i \in \{1,2\}$, and let $e=\{u,v\} \not \in E(G)$.  Then $G+e$ admits a barbell partition provided $\left\vert N_G(u) \cap W_{i} \right\vert \geq 2$.
\end{obs}

%\textcolor{blue}{define $G+e$?}

The \emph{lollipop} graph, $L_{n,l}$ for $n\geq 2$ and $l\geq 1$, is defined to be a complete graph $K_n$ together with a path $P_l$ such that a leaf of $P_l$ and a vertex of $K_n$ are adjacent, and no other extra edges between the vertices of $K_n$ and $P_l$ exist. In \cite{SSPforGs} it is shown that $L_{n,l} \in G^{SSP}$.

\begin{prop}
If the pendant vertex of the lollipop graph $L_{n,1}$ is duplicated without an edge (call the new graph $G$), and  $A\in \mathcal{S}(K_n)$ then the SSP is preserved for $B \in \mathcal{S}(G)$ such that 
\[B = \begin{bmatrix}
A & \textbf{e}_{n} & \textbf{e}_{n}\\
\textbf{e}_{n}^T & \mu_1 & 0\\ \textbf{e}_{n}^T & 0 & \mu_2\\
\end{bmatrix},\] 
where $\mu_1 \neq \mu_2 \in \mathbb{R}$ and $e_n$ represents a standard basis vector with a 1 in the $n$th coordinate. 
%for $1 \leq j \leq 2$ and $\lambda_j \neq 0$. \textcolor{blue}{In ${e}_{\lambda_j}$ is $j$ representing the location of the nonzero entry or the index  $1 \leq j \leq 2$?}
\end{prop}

\begin{proof}
%Notice from Example 5.7 in \cite{SSPforGs} that $L_{n,1} \in G^{SSP}$.\textcolor{blue}{First sentence is not needed here.} 
From the first two requirements of SSP, symmetric $X$ such that $B \circ X = O$ and $I \circ X = O$ yields that \[X = \begin{bmatrix}
O & \textbf{x}_1 & \textbf{x}_2\\
\textbf{x}_1^T & 0 & y\\ 
\textbf{x}_2^T & y & 0\\
\end{bmatrix},\] 
where $y \in \mathbb{R}$ and  the $n$th coordinate of $\textbf{x}_j$ is necessarily 0 for $j=1,2$.
%has only one possible nonzero entry in the in the $s$th coordinate \textcolor{blue}{or just $x_{is}=0$, for $i\in \{1,2\}$}. 
Then, the third condition yields $BX = XB$ where
\[BX = \begin{bmatrix}
A & \textbf{e}_{n} & \textbf{e}_{n}\\
\textbf{e}_{n}^T & \mu_1 & 0\\ \textbf{e}_{n}^T & 0 & \mu_2\\
\end{bmatrix} \begin{bmatrix}
O & \textbf{x}_1 & \textbf{x}_2\\
\textbf{x}_1^T & 0 & y\\ 
\textbf{x}_2^T & y & 0\\
\end{bmatrix} = \begin{bmatrix}
\textbf{e}_{n}\textbf{x}_1^T + \textbf{e}_{n}\textbf{x}_2^T & A\textbf{x}_1+y\textbf{e}_{n} & A\textbf{x}_2+y\textbf{e}_{n}\\
\mu_1\textbf{x}_1^T & \textbf{e}_{n}^T\textbf{x}_1 & \textbf{e}_{n}^T\textbf{x}_2+\mu_1y\\ 
\mu_2\textbf{x}_2^T & \textbf{e}_{n}^T\textbf{x}_1+\mu_2y & \textbf{e}_{n}^T \textbf{x}_2\\
\end{bmatrix}\] 
\[XB =  \begin{bmatrix}
O & \textbf{x}_1 & \textbf{x}_2\\
\textbf{x}_1^T & 0 & y\\ 
\textbf{x}_2^T & y & 0\\
\end{bmatrix}\begin{bmatrix}
A & \textbf{e}_{n} & \textbf{e}_{n}\\
\textbf{e}_{n}^T & \mu_1 & 0\\ \textbf{e}_{n}^T & 0 & \mu_2\\
\end{bmatrix} = \begin{bmatrix}
\textbf{x}_1\textbf{e}_{n}^T + \textbf{x}_2\textbf{e}_{n}^T & \mu_1\textbf{x}_1 & \mu_2\textbf{x}_2\\
\textbf{x}_1^TA+y\textbf{e}_{n}^T & \textbf{x}_1^T\textbf{e}_{n} & \textbf{x}_1^T\textbf{e}_{n}+\mu_2y\\ 
\textbf{x}_2^TA+y\textbf{e}_{n}^T & \textbf{x}_2^T\textbf{e}_{n}+\mu_1y & \textbf{x}_2^T\textbf{e}_{n}\\
\end{bmatrix}.\] 
Simplifying: 
\[\begin{bmatrix}
\textbf{e}_{n}\textbf{x}_1^T + \textbf{e}_{n}\textbf{x}_2^T & A\textbf{x}_1+y\textbf{e}_{n} & A\textbf{x}_2+y\textbf{e}_{n}\\
\mu_1\textbf{x}_1^T & \overline{0} & \overline{0}+\mu_1y\\ 
\mu_2\textbf{x}_2^T & \overline{0}+\mu_2y & \overline{0}\\
\end{bmatrix} = \begin{bmatrix}
\textbf{x}_1\textbf{e}_{n}^T + \textbf{x}_2\textbf{e}_{n}^T & \mu_1\textbf{x}_1 & \mu_2\textbf{x}_2\\
\textbf{x}_1^TA+y\textbf{e}_{n}^T & \overline{0} & \overline{0}+\mu_2y\\
\textbf{x}_2^TA+y\textbf{e}_{n}^T & \overline{0}+\mu_1y & \overline{0}\\
\end{bmatrix}.\] 
Comparing the (2,3) blocks above we conclude that $y=0$.  Further, 
$\textbf{x}_1\textbf{e}_{n}^T + \textbf{x}_2\textbf{e}_{n}^T = 
\textbf{e}_{n}\textbf{x}_1^T + \textbf{e}_{n}\textbf{x}_2^T$, together with $x_{1n}=x_{2n}=0$, imply
 $-x_{1j} = x_{2j}$, where $1 \leq j \leq n$ and $j \neq i$. Hence  $\textbf{x}_1$ is a multiple of $\textbf{x}_2$. Now comparing the (1,2) and the (1,3) blocks above we may conclude that $\textbf{x}_1= \textbf{x}_2=0$. Hence, $X = O$. 
\end{proof}

We next turn our attention to the duplication or removal of a vertex.

\begin{obs}
Let $G$ be a graph admitting a barbell partition $\{R, W_1, W_2\}$, and let $S \subset V(G)$.  If $S \subseteq R$, then $G-S$ admits a barbell partition.
\end{obs}

\begin{obs}
Let $G$ be a graph admitting a barbell partition $\{R, W_1, W_2\}$.  If $v \in W_i$ for some $i \in \{1,2\}$ and either
\begin{itemize}
    \item $N_G(v) \subset W_i$ or
    \item for every $u \in \left(R \cap N_G(v)\right)$, we have that $\left\vert N_G(u) \cap W_i \right\vert>2$,
\end{itemize}
then $G-v$ admits a barbell partition.
\end{obs}

The first case is immediate because when a vertex and its neighborhood lie in the same element of the barbell partition none of the criteria are effected by removing the vertex.  The second case follows because if for every $u \in \left(R \cap N_G(v)\right)$, we have that $\left\vert N_G(u) \cap W_i \right\vert>2$, then for each such $u$ we have $\left\vert N_{G-v}(u) \cap W_i \right\vert \geq 2$.

%\textcolor{blue}{not defined}

\begin{obs}\label{relpen}
Let $G$ be a graph admitting a barbell partition $\{R, W_1, W_2\}$, and suppose $G=H-v$. If $X \in \{R,W_1,W_2\}$ and $N_H(v) \subset X$, then $H$ admits a barbell partition.
\end{obs}

\begin{obs}\label{obsvinr}
Let $G$ be a graph admitting a barbell partition $\{R, W_1, W_2\}$, and suppose $G=H-v$. If for each $i \in \{1,2\}$, $\left\vert N_H(v) \cap W_i \right\vert \neq 1$, then $H$ admits a barbell partition.
\end{obs}

%\textcolor{blue}{Can this move to the statement? Letting $R'=R \cup \{v\}$ it follows that $\{R',W_1,W_2\}$ is a barbell partition of $H$.}

\begin{obs}\label{obsvinw}
Let $G$ be a graph admitting a barbell partition $\{R, W_1, W_2\}$, and suppose $G=H-v$.  If for some $i \in \{1,2\}$, we have that for every $u \in R \cap N_H(v)$, $\left\vert N_G(u) \cap W_i \right\vert \geq 2$ and $N_H(v) \setminus (R \cup W_i)=\emptyset$, then $H$ admits a barbell partition.
\end{obs}

Supposing without loss of generality that the $i$ mentioned in the above observation is $1$, then letting $W_1'=W_1 \cup \{v\}$ it follows that $\{R,W_1',W_2\}$ is a barbell partition of $H$.

Having considered the effect of adding or removing vertices or edges on the presence of barbell partitions of a graph, we now consider the more specific graph operation of vertex duplication and provide results establishing the interaction between barbell partitions and vertex duplication.

\begin{thm}\label{dupjdup}
Let $G$ be a graph admitting a barbell partition and $v \in V(G)$.  Let $H=dup(G,v)$  and let $K=jdup(G,v)$. Then both $H$ and $K$ admit barbell partitions, and in particular neither graph is a member of $G^{SSP}$.
\end{thm}

\begin{proof}
Let $u$ be the duplication of $v$, and let $\{R,W_1,W_2\}$ be a barbell partition of $G$.  

\noindent \textbf{Case 1:} $v \in R$.

Since $v \in R$, it follows that $\abs{N_G(v) \cap W_i} \neq 1$ for each $i \in \{1,2\}$.  Since $N_H(u)=N_G(v)$ and $N_K(u)=N_G[v]$, it follows that $\abs{N_H(u) \cap W_i}=\abs{N_K(u) \cap W_i}\neq 1$ for each $i \in \{1,2\}$.  Thus, by Observation \ref{obsvinr}, $H$ and $K$ each have barbell partitions, specifically $\{R',W_1, W_2\}$ with $R'=R \cup \{u\}$.

%Since $W_1$ and $W_2$ are unchanged, they are still nonempty; and since the only edges which were added are incident to $u$ and $u \in R'$, there are still no edges between $W_1$ and $W_2$.  Note, if $w \in R' \setminus \{u\}$, then for $i \in \{1,2\}$, $\left\vert N_H(w) \cap W_i \right\vert=\left\vert N_G(w) \cap W_i \right\vert$.  Since $N_H(u)=N_H(v)=N_G(v)$, for each $i \in \{1,2\}$, we have $\left\vert N_H(u) \cap W_i \right\vert=\left\vert N_G(v) \cap W_i \right\vert$, and since $\{R,W_1,W_2\}$ is a barbell partition of $G$ we know that $\left\vert N_H(u) \cap W_i \right\vert \neq 1$.  Thus $\{R',W_1',W_2'\}$ is a barbell partition of $H$.

\noindent \textbf{Case 2:} There exists $i \in \{1,2\}$ such that $v \in W_i$.

Without loss of generality, let $i=1$.  So, $N_G[v] \cap W_2=\emptyset$.  Again, since $N_H(u)=N_G(v)$, it follows that $N_H(u) \cap W_2=\emptyset$.  Since $v \in W_1$, for every $r \in R \cap N_G(v)$, $\left\vert N_G(r) \cap W_1 \right\vert \geq 2$.  So, for every $r \in R \cap N_H(u)$, $\left\vert N_H(r) \cap W_1 \right\vert \geq 3$.  Thus, by Observation \ref{obsvinw}, $H$ has a barbell partition, specifically $\{R,W'_1, W_2\}$ with $W_1'=W_1 \cup \{u\}$.

Since $N_K(u)=N_G[v]$, by similar reasoning $K$ has a barbell partition.
\end{proof}

%Since $W_1$ and $W_2$ are at least as large as they were, they are still nonempty; and since no edge incident a vertex in $W_2$ has been added, there are still no edges between $W_1$ and $W_2$.  For every vertex $w \in R$, either $\{w,v\}\not\in E(G)$ and $\left\vert N_H(w) \cap W_i \right\vert=\left\vert N_G(w) \cap W_i \right\vert$ or $\{w,v\} \in E(G)$ and $\left\vert N_H(w) \cap W_i \right\vert=\left\vert N_G(w) \cap W_i \right\vert+1$.  In the latter case, since $w \in R$ and $w$ has neighbors in $W_1$, $\left\vert N_G(w) \cap W_1 \right\vert \geq 2$, and so $\left\vert N_H(w) \cap W_1 \right\vert \neq 1$.  Thus $\{R',W'_1,W'_2\}$ is a barbell partition of $H$.

%Since $V(K)=V(H)$ and $E(K)=E(H) \cup \{u,v\}$ and in each case $u$ is taken to be in the same member of the partition as $v$, the proof that $K$ admits a barbell partition is identical.

\noindent \underline{Note:} The converse of Theorem \ref{dupjdup} is not true.  For example, $K_{1,3}$ does not admit a barbell partition, but both duplicating and join-duplicating a leaf will yield a graph which admits a barbell partition.  This observation inspires the following results.

We now explore a little further the connection between forts in a graph and barbell partitions, both concepts, of course, are of interest to the IEPG.

\begin{lem}\label{dupfort}
Let $G$ be a graph, let $v \in V(G)$, and let $G'$ be the graph yielded by (join) duplicating $v$.  Then $\{v,v'\}$ is a fort of $G$.
\end{lem}

\begin{proof}
Since no vertex outside of $N_G[v]$ is neighbors with $v'$ and vice versa, $\{v,v'\}$ is a fort of $G'$.
\end{proof}

\begin{lem}\label{fortbarbell}
Let $G$ be a graph and $F \subset V(G)$ be a fort of $G$.  If $N_G[F]$ is not a zero forcing set of $G$ then $G$ admits a barbell partition.
\end{lem}

\begin{proof}
Since $N_G[F]$ is not a zero forcing set of $G$, by Theorem \ref{fort}, $V(G) \setminus N_G[F]$ contains a fort of $G$, call it $F'$.  Since $F' \subseteq V(G)\setminus N_G[F]$, $F \cap F'=\emptyset$ and there do not exist vertices $v,v' \in V(G)$ such that $v \in F$, $v \in F'$, and $vv' \in E(G)$.  So $\{R,F,F'\}$ forms a barbell partition of $G$, where $R=V(G)\setminus (F \cup F')$. 
\end{proof}

\begin{obs}\label{fortsubg}
Let $G$ be a graph and $H$ be a vertex induced subgraph  of $G$.  If $S \subsetneq V(H)$ does not contain a fort of $H$, then $S$ does not contain a fort of $G$. 
\end{obs}

\setcounter{case}{0}

\begin{thm}\label{fortbarbthm}
Let $G$ be a graph which does not admit a barbell partition.  Let $v \in V(G)$, and let $G'$ be the graph yielded by (join) duplicating $v$.  Then $G'$ admits a barbell partition if and only if $V(G) \setminus N_G[v]$ contains a fort of $G$.
\end{thm}

\begin{proof}
First suppose $V(G) \setminus N_G[v]$ contains a fort of $G$, call it $F$. Let $v'$ be the new vertex that duplicates $v$. Note, by Lemma \ref{dupfort}, $\{v,v'\}$ is a fort of $G'$.  Next, since $F \subseteq V(G) \setminus N_G[v]$ and $N_{G'}(v') \subseteq N_G[v]$, it follows that $N_{G'}[v'] \cap F=\emptyset$ and so $F$ is a fort of $G'$ contained in $V(G') \setminus N_{G'}[\{v,v'\}]$.  So by Theorem \ref{fort}, $N_{G'}[\{v,v'\}]$ is not a zero forcing set of $G'$.  Thus by Lemma \ref{fortbarbell}, $G'$ admits a barbell partition.

Next, suppose $V(G) \setminus N_G[v]$ does not contain a fort of $G$.  If $G'$ does not have a pair of separated forts, then we are done, so, suppose that $G'$ has a pair of separated forts, $\{F_1,F_2\}$.  Since $V(G) \setminus N_G[v]$ does not contain a fort of $G$, by Observation \ref{fortsubg}, $V(G') \setminus N_{G'}[\{v,v'\}]$ does not contain a fort of $G'$.  So, any fort of $G'$ must contain a member of $N_{G'}[\{v,v'\}]$.  Let $N=N_{G'}[\{v,v'\}]\setminus \{v,v'\}$, and note that $N=N_G(v)$.

\begin{case}
Either $F_1 \cap \{v,v'\} \neq \emptyset$ and $F_2 \cap N \neq \emptyset$ or $F_1 \cap N \neq \emptyset$ and $F_2 \cap \{v,v'\} \neq \emptyset$.
\end{case}

Then there exist vertices $u_1 \in F_1$ and $u_2 \in F_2$ such that $u_1u_2 \in E(G')$, and so $\{F_1,F_2\}$ is not a pair of separated forts of $G'$, a contradiction.

\begin{case}
$F_i \cap N \neq \emptyset$ and $F_i \cap \{v,v'\} = \emptyset$, for each $i \in \{1,2\}$.
\end{case}

Then $\{F_1,F_2\}$ is a pair of separated forts in $G'$ if and only if $\{F_1,F_2\}$ is a pair of separated forts in $G$.  Since $G$ does not admit a barbell partition, $\{F_1,F_2\}$ is not a pair of separated forts of $G'$, a contradiction.

\begin{case}
$F_i \cap \{v,v'\} \neq \emptyset$ and $F_i \cap N = \emptyset$, for each $i \in \{1,2\}$.
\end{case}

Let $F=(F_1 \cup F_2) \setminus \{v,v'\}$.  To reach our contradiction we will show that $F$ is a fort contained in $G \setminus N_G[v]$. First, since $F_1 \cap F_2 = \emptyset$ but $F_i \cap \{v,v'\} \neq \emptyset$, for each $i \in \{1,2\}$, it must be that a unique element of $\{v,v'\}$ is a member of $F_1$ and that the other element is a member of $F_2$.  Suppose without loss of generality that $v \in F_1$ and $v' \in F_2$.  Since $\{F_1,F_2\}$ is a pair of separated forts of $G'$ it follows that no vertex in $V(G') \setminus (F_1 \cup F_2)$ is adjacent to exactly one element of either $F_1$ or $F_2$.  Thus no vertex in $V(G') \setminus (F_1 \cup F_2)$ is adjacent to exactly one element of $F_1 \cup F_2$.  Since $N_{G'}(u)=N_G(u)$ for every vertex $u \in V(G) \setminus N_G[v]$, $\abs{N_G(u) \cap F} \neq 1$ for every vertex $u \in V(G) \setminus N_G[v]$.  So, it simply remains to check the vertices in the set $N_G[v]$.  Since $F_1 \cap N=\emptyset$ but each vertex of $N$ is adjacent to $v$, each vertex of $N$ must be neighbors with some vertex in $F_1 \setminus \{v\}$.  Likewise, each vertex of $N$ must be neighbors with some vertex in $F_2 \setminus \{v'\}$.  So each vertex in $N$ must be neighbors with at least two vertices in $(F_1 \cup F_2) \setminus \{v,v'\}$, and thus each vertex in $N_G(v)$ must be neighbors with at least two vertices in $F$.  Finally, since $(F_1 \cup F_2) \cap N=\emptyset$, it follows that $v$ is not neighbors with any vertex in $F$.  So, for each vertex $u \in V(G) \setminus F$, $\abs{N_G(u) \cap F} \neq 1$.  Thus $G \setminus N_G[v]$ contains a fort, a contradiction.

\vspace{0.1in}

\noindent In each case we reach a contradiction, and thus $G'$ does not have a pair of separated forts.  So by Observation \ref{fortbarb}, $G'$ does not admit a barbell partition completing the proof. 
\end{proof}

%\textcolor{red}{(We might want to include the following in the paper, but we would first need to revise it following a short discussion on whether anyone sees any other graphs or items to mention on this subject.)}
%\textcolor{blue}{(Theorem 3.13- suggests that it might be worth looking at duplication of vertices $v$ in graphs $G$ such that $V(G) \setminus N_G[v]$ does not contain a fort, since the lack of a barbell partition in $G'$ does not ensure that $G' \in G^{SSP}$.  I believe this collection of graphs is certainly worth exploration.  However, it is also worth noting that there are examples of graphs for which $V(G) \setminus N_G[v]$ does not contain a fort but $G' \not \in  G^{SSP}$.  To see one such example, consider Figure 4b p89 of "The Strong Spectral Property for Graphs".  This graph is the result of duplicating the third vertex of $P_5$.  However, in that case $V(G) \setminus N_G[v]$ is just the two endpoints of the path and the second and fourth vertices are adjacent to exactly one endpoint meaning the endpoints do not form a fort.)}

\begin{prop}\label{ssp-pathprop}
Let $n \geq 2$. Let $v \in V(P_n)$ be a pendant vertex.  Let $G$ be the graph yielded by join duplication of $v$.  Then $G \in  G^{SSP}$.
\end{prop}

\begin{proof}
If we join duplicate a pendant vertex of $P_n$, then we obtain a graph $H$ with the property that $q(H)=|H|-1$. For this class of graphs it is known that $H \in G^{SSP}$ (see \cite{SSPforGs}).
\end{proof}

\begin{thm}\label{ssp-fortthm}
Let $G$ be a graph and suppose $v \in V(G)$ is a pendant vertex of $G$.  Let $G'$ be the graph yielded by (join) duplicating $v$.  Then $G' \in G^{SSP}$ if and only if $G$ is a path. 
\end{thm}

\begin{proof}
First, by Proposition \ref{ssp-pathprop} it follows that if $G$ is a path, then $G' \in  G^{SSP}$.

Now suppose, $G$ is not a path.  Since $v$ is a pendant vertex, we can let $N_G(v)=\{u\}$.  Since $G$ is not a path and $v$ is a pendant vertex of $G$, it follows that either $G-v$ is a path and $u$ is not an endpoint of $G-v$ or $G-v$ is not a path.  In either case, $\{u\}$ is not a zero forcing set of $G-v$.  Thus by Theorem \ref{fort} $V(G-v)\setminus \{u\}$ contains a fort of $G-v$, call it $F$.  Since $N_G[v]=\{u,v\}$, it follows that $F$ is a fort of $G$ contained in $V(G) \setminus N_G[v]$.  Thus, by Theorem \ref{fortbarbthm}, $G'$ admits a barbell partition.  Finally, it follows that $G' \not \in  G^{SSP}$.
\end{proof}

\section{Barbell Partitions, Vertex Sums, and Joins}

In this section we consider barbell partitions associated with some standard graph operations (namely, vertex sums and joins), and we begin with the following result concerning a basic necessary condition for the existence of a barbell partition in a graph.

%\textcolor{blue}{I think in the following lemma $G$ should be connected as $\overline{K_2}$ has a barbell partition. Then the statement is kind of obvious: 1) if $R$ is empty the graph is disconnected, 2) if $R$ is not empty then $W_i$ can't be just one vertex as otherwise there would be only one neighbor from $R$ to it.} 

\begin{lem}\label{noiso}
Let $G$ be a graph with no isolated vertices.  If $\{R,W_1,W_2\}$ is a barbell partition of $G$, then $\abs{W_1},\abs{W_2} \geq 2$.
\end{lem}

\begin{proof}
Let $w_1 \in W_1$ and $w_2 \in W_2$.  First suppose that the component of $G$ containing $w_1$ contains no members of $R$.  Since this means every vertex in this component is either in $W_1$ or $W_2$, it follows that every vertex in this component must be in $W_1$, otherwise there exists vertices $u,v$ with $u \in W_1$, $v \in W_2$, and $uv \in E(G)$ which would imply that $\{R,W_1,W_2\}$ is not a barbell partition of $G$.  Since $G$ has no isolated vertices, there must be another vertex $w_1' \in W_1$ in this component.  So, $\abs{W_1} \geq 2$.

Now suppose the component of $G$ containing $w_1$ contains at least one member of $R$.  Since this component is connected and contains members of both $R$ and $W_1$, there must exist a pair of vertices $r \in R$ and $w'_1 \in W_1$ such that $rw'_1 \in E(G)$.  Since $\left\vert N_G(r) \cap W_1 \right\vert \neq 0$ and $\{R,W_1,W_2\}$ is a barbell partition of $H$, it follows that $\left\vert N_G(r) \cap W_1 \right\vert \geq 2$, and thus $|W_1| \geq 2$.

An identical argument shows that $|W_2| \geq 2$.
\end{proof}

It is of course not necessary for a graph $G$ to not have any isolated vertices for it to admit a barbell partition $\{R,W_1,W_2\}$ for which $\abs{W_1},\abs{W_2} \geq 2$, as witnessed by the graph $G$ with vertex set $V(G)=\{v_i\}_{i=1}^4$ and edge set $E(G)=\emptyset$.  However, this can be viewed as establishing that this property occurs anytime a graph $G$ possesses a barbell partition such that neither $W_1$ nor $W_2$ is a single isolated vertex.  It also provides the following biconditional result concerning the interaction between the join of two graphs and barbell partitions.

\begin{thm}
Let $G$ and $H$ each be graphs with no isolated vertices and $K=G \vee H$.  Then $K$ admits a barbell partition, if and only if either $G$ or $H$ admits a barbell partition $\{R,W_1,W_2\}$ for which $\abs{W_1},\abs{W_2} \geq 2$.
\end{thm}

%\begin{figure}[h]
  %\centering
  %\includegraphics[width=0.8\linewidth]{join.png}
  %\label{join-0}
%\end{figure}

\begin{proof}
First suppose $\{R,W_1,W_2\}$ is a barbell partition of $H$ for which $\abs{W_1},\abs{W_2} \geq 2$, and let $R'=R \cup V(G)$.  We will show that $\{R',W_1,W_2\}$ is a barbell partition of $K$.

\vspace{0.1in}
\begin{addmargin}[0.87cm]{0cm}{\bf Claim 1:}
There do not exist vertices $w_1 \in W_1$ and $w_2 \in W_2$ such that $w_1w_2 \in E(K)$.
\vspace{0.1in}

\noindent {\em Proof of Claim 1.}
Note, for vertices $u,v \in V(H)$, $uv \in E(H) \Longleftrightarrow uv \in E(K)$.  Since $\{R,W_1,W_2\}$ is a barbell partition of $H$, it follows that there do not exist vertices $w_1 \in W_1$ and $w_2 \in W_2$ such that $w_1w_2 \in E(K)$.
\end{addmargin}

\vspace{0.1in}

\begin{addmargin}[0.87cm]{0cm}{\bf Claim 2:}
For each $r \in R'$, $\left\vert N_K(r) \cap W_i \right\vert \neq 1$ for $i \in \{1,2\}$.
\vspace{0.1in}

\noindent {\em Proof of Claim 2.} 
Again, for vertices $u,v \in V(H)$, $uv \in E(H) \Longleftrightarrow uv \in E(K)$.  Since $\{R,W_1,W_2\}$ is a barbell partition of $H$, it follows that for every vertex $r \in R$ and each $i \in \{1,2\}$, $\left\vert N_K(r) \cap W_i\right\vert \neq 1$.  Next, since $K=G \vee H$ and $W_1 \cup W_2 \subset V(H)$, every vertex $v \in V(G)$ is adjacent to every vertex in $W_1 \cup W_2$.  Since $\abs{W_1},\abs{W_2} \geq 2$ for each $v \in V(G)$ and each $i \in \{1,2\}$, $\left\vert N_K(v) \cap W_i \right\vert \geq 2$.  Since $R'=R \cup V(G)$, it follows that for every vertex $r \in R'$ and for each $i \in \{1,2\}$, $\left\vert N_K(r) \cap W_i \right\vert \neq 1$.
\end{addmargin}
\vspace{0.1in}

Thus $\{R', W_1, W_2\}$ forms a barbell partition of $K=G \vee H$.

\vspace{0.1in}

Next suppose $\{R,W_1,W_2\}$ is a barbell partition of $K=G \vee H$.  Since for every pair of vertices $g \in V(G)$ and $h \in V(H)$, we have $gh \in E(K)$, it follows that for each $i \in \{1,2\}$, 
\[W_i \cap V(G) \neq \emptyset \Longrightarrow W_{3-i} \subset V(G) \text{ and } W_i \cap V(H) \neq \emptyset \Longrightarrow W_{3-i} \subset V(H).\]

%\textcolor{blue}{\[W_i \cap V(G) \neq \emptyset \Longrightarrow W_{3-i} \subset V(G) \text{ and } W_i \cap V(H) = \emptyset \Longrightarrow W_{3-i} \subset V(G).\]}

Since $W_1,W_2 \neq \emptyset$, it must be that either $W_1 \cup W_2 \subseteq V(G)$ or $W_1 \cup W_2 \subseteq V(H)$.  Without loss of generality, suppose that $W_1 \cup W_2 \subseteq V(H)$.  It thus follows that $V(G) \subseteq R$.  Finally, since for vertices $u,v \in V(H)$, $uv \in E(H) \Longleftrightarrow uv \in E(K)$, it follows that there do not exist vertices $w_1 \in W_1$ and $w_2 \in W_2$ such that $w_1w_2 \in E(H)$ and for each $r \in R \cap V(H)$ and each $i \in \{1,2\}$, it also follows that $\left\vert N_H(r) \cap W_i\right\vert \neq 1$.  Thus $\{R \cap V(H), W_1,W_2\}$ is a barbell partition of $H$.  Finally, since $H$ has no isolated vertices, by Lemma \ref{noiso} $\abs{W_1},\abs{W_2} \geq 2$.
\end{proof}

%\textcolor{blue}{here $G+v$ is used for dominating vertex $v$ but I think before (like page 8) it is used differently.}

\begin{cor}\label{domin}
Let $G$ be a graph with no isolated vertices, and let $H$ be the graph obtained from $G$ by adding a dominating vertex $v$.  If $\{R,W_1,W_2\}$ is a barbell partition of $G$, then $\{R \cup \{v\},W_1,W_2\}$ is a barbell partition of $H$.
\end{cor}

\begin{proof}
Since $\{R,W_1,W_2\}$ is a barbell partition of $G$ and given $u,v \in V(G)$, $uv \in E(H)$ if and only if $uv \in E(G)$, it follows that there do not exist vertices $w_1 \in W_1$, $w_2 \in W_2$ such that $w_1w_2 \in E(H)$.  Similarly, it follows that for each $r \in R$, $\abs{N_H(r) \cap W_i} \neq 1$ for each $i \in \{1,2\}$. Since $G$ has no isolated vertices, by Lemma \ref{noiso}, it follows that $\abs{W_1},\abs{W_2} \geq 2$.  Furthermore, since $v$ is adjacent to every vertex in $W_1 \cup W_2$ and $\abs{W_1},\abs{W_2} \geq 2$, it follows that $\abs{N_H(v) \cap W_i} \neq 1$ for each $i \in \{1,2\}$.  Thus $\{R \cup \{v\},W_1,W_2\}$ is a barbell partition of $H$.
\end{proof}

Let $G$ and $H$ be two graphs. The graph obtained from $G$ and $H$ by identifying a vertex $v$ in both $G$ and $H$ is called the vertex sum of $G$ and $H$ at $v$ and is denoted by $G \oplus_v H$. Observe that $v$ is necessarily a cut vertex of $G \oplus_v H$. Along these lines, it follows as an immediate corollary of Theorem 4.3 and Corollary 2.4 in \cite{SSPforGs} that, for two path graphs $P_n$ and $P_m$, $P_n \oplus_v P_m$ admits a barbell partition if and only if $\deg_{P_n}(v)=\deg_{P_m}(v)=2$.  The next result is concerned with the vertex sum of two graphs excluding paths.

%\begin{proof}
%If $\deg_G(v)=\deg_H(v)=2$, then $G \oplus_v H$ is a spider with 4 legs.  Letting $R=\{v\}$ and each of $W_1$ and $W_2$ be two of the components of $G-v$, clearly forms a barbell partition.

%On the other hand if either $\deg_G(v) \neq 2$ or $\deg_H(v) \neq 2$, then $G \oplus_v H$ will either be a path or a spider with 3 legs.  Since each path is a member of $G^{SSP}$, a path does not admit a barbell partition.  Furthermore, any fort of a spider with 3 legs must contain the final vertex of at least 2 of the spider's legs.  Thus, a spider with 3 legs cannot have two disjoint forts, and thus cannot have a barbell partition.
%\end{proof}

%\textcolor{blue}{Should we say something like ``for any vertex $v$'' in the next two theorems? I tried to add something about the vertex but since we can't write $v\in V(G)\cap V(H)$ I am not sure what is best.}

\begin{thm}\label{oplus-thm1}
Let $G$ and $H$ be graphs which are not paths.  Then $G \oplus_v H$, where $v$ is any identified vertex of both $G$ and $H$, admits a barbell partition.
\end{thm}

%\textcolor{blue}{Here we need to mention that $P_1$ is also a path, so I added that to the definitions, but we don't need this for $C_n$ and $k_n$, please check.}

\begin{proof}
Since $G$ is not a path, it follows that $\{v\}$ cannot be a zero forcing set of $G$.  Thus, by Theorem \ref{fort}, $V(G) \setminus \{v\}$ must contain a fort $F_G$ of $G$.  Likewise, since $H$ is not a path, $V(H) \setminus \{v\}$ must contain a fort $F_H$ of $H$.  Since $v$ is the only vertex in $H$ with neighbors in $G$, and vice versa, it follows that $\{F_G,F_H\}$ is a pair of separated forts of $G \oplus_v H$, completing the proof.\end{proof}

\begin{thm}\label{oplus-thm2}
Let $G$ be a graph and $H$ be a graph admitting a barbell partition.  Then $K=G \oplus_v H$, where $v$ is any identified vertex of both $G$ and $H$,  admits a barbell partition.
\end{thm}

%\begin{figure}[h]
  %\centering
  %\includegraphics[width=0.6\linewidth]{vertexsum.png}
%\end{figure}

\begin{proof}
Let $\{R, W_1, W_2\}$ be a barbell partition of $H$.

\noindent \textbf{Case 1:} $v \in R$

Let $R'=R \cup V(G)$.  We will show that $\{R',W_1,W_2\}$ is a barbell partition of $K=G \oplus_v H$. 

\vspace{0.1in}

\begin{addmargin}[0.87cm]{0cm}{\bf Claim 1.1:}
$W_1,W_2 \not = \emptyset$ and there do not exist vertices $w_1 \in W_1$ and $w_2 \in W_2$ such that $\{w_1,w_2\} \in E(K)$.
\vspace{0.1in}

\noindent {\em Proof of Claim 1.1.}
Since $\{R,W_1,W_2\}$ is a barbell partition of $H$, it follows that $W_1,W_2 \neq \emptyset$ and there are no edges between vertices in $W_1$ and $W_2$ in $H$.  Furthermore, since $V(G) \subset R'$, there are no edges between $W_1$ and $W_2$ in $K$.
\end{addmargin}

\vspace{0.1in}

\begin{addmargin}[0.87cm]{0cm}{\bf Claim 1.2:}
For each $r \in R'$ and $i \in \{1,2\}$, we have that $\left\vert N_{K}(r) \cap W_i \right\vert \neq 1$. 
\vspace{0.1in}

\noindent {\em Proof of Claim 1.2.} If $r \in R' \cap V(H)$, since $\{R,W_1,W_2\}$ is a barbell partition of $H$ and $V(G) \subseteq R'$, for each $i \in \{1,2\}$, we have $\left\vert N_K(r) \cap W_i \right\vert = \left\vert N_H(r) \cap W_i \right\vert \not = 1$. If $r \notin R' \cap V(H)$, since no vertex in $V(G) \setminus \{v\}$ is adjacent to a vertex in $V(H) \setminus \{v\}$, it follows that $\left\vert N_K(r) \cap W_i \right\vert =0$.
\end{addmargin}
\vspace{0.1in}

\noindent Thus $\{R',W_1,W_2\}$ is a barbell partition of $K=G \oplus_v H$.

%\noindent {\em Proof of Claim 1.2.}
%Since $\{R,W_1,W_2\}$ is a barbell partition of $H$ and $V(G) \subseteq R'$, for each $u \in R \cap V(H)$ and each $i \in \{1,2\}$, we have $\left\vert N_K(u) \cap W_i \right\vert = \left\vert N_H(u) \cap W_i \right\vert \not = 1$.  Since no vertex in $V(G) \setminus \{v\}$ is adjacent a vertex in $V(H) \setminus \{v\}$, it follows that for each $u \in V(G)\setminus \{v\}$ and each $i \in \{1,2\}$, we have that $\left\vert N_K(u) \cap W_i \right\vert =0$.
%\end{addmargin}
%\vspace{0.1in}
%
%\textcolor{blue}{How about this: \noindent {\em Proof of Claim 1.2.} If $r \in R' \cap V(H)$, since $\{R,W_1,W_2\}$ is a barbell partition of $H$ and $V(G) \subseteq R'$, for each $i \in \{1,2\}$, we have $\left\vert N_K(r) \cap W_i \right\vert = \left\vert N_H(r) \cap W_i \right\vert \not = 1$. If $r \notin R' \cap V(H)$, since no vertex in $V(G) \setminus \{v\}$ is adjacent a vertex in $V(H) \setminus \{v\}$, it follows that $\left\vert N_K(r) \cap W_i \right\vert =0$.}
%
%\noindent Thus $\{R',W_1,W_2\}$ is a barbell partition of $K=G \oplus_v H$.

\vspace{0.1in}

\noindent \textbf{Case 2 :}
Assume $v \in W_i$ for some $i \in \{1,2\}$

Without loss of generality suppose $v \in W_1$, and let $W_1'=W_1 \cup V(G)$.  We will show that $\{R,W'_1,W_2\}$ is a barbell partition of $K=G \oplus_v H$. 

\begin{addmargin}[0.87cm]{0cm}{\bf Claim 2.1:}
$W'_1,W_2 \not = \emptyset$ and there do not exist vertices $w_1 \in W'_1$ and $w_2 \in W_2$ such that $w_1w_2 \in E(K)$.
\vspace{0.1in}

\noindent {\em Proof of Claim 2.1.}
Since $W_1 \subset W_1'$, it follows that $W_1'$ and $W_2$ are nonempty. Since there are no edges between $W_1$ and $W_2$ and no vertex in $V(G) \setminus \{v\}$ is adjacent to a vertex in $V(H) \setminus \{v\}$, it follows that there are no edges between $W_1'$ and $W_2$.
\end{addmargin}

\vspace{0.1in}

\begin{addmargin}[0.87cm]{0cm}{\bf Claim 2.2:}
For each $r \in R$, we have that $\left\vert N_{K}(r) \cap W'_1 \right\vert \neq 1$ and $\left\vert N_{K}(r) \cap W_2 \right\vert \neq 1$.
\vspace{0.1in}

\noindent {\em Proof of Claim 2.2.}
Since $V(G) \subset W_1'$ and no vertex in $V(H) \setminus \{v\}$ is adjacent to a vertex in $V(G) \setminus \{v\}$, it follows that for each $r \in R$, we have that $\left\vert N_{K}(r) \cap W'_1 \right\vert \neq 1$ and $\left\vert N_{K}(r) \cap W_2 \right\vert \neq 1$.  
\end{addmargin}
\vspace{0.1in}
\noindent Thus $\{R, W_1',W_2\}$ is a barbell partition of $K=G \oplus_v H$.
\end{proof}

From the above results we have the following straightforward consequence.

\begin{obs}
    We can conclude from Theorem \ref{oplus-thm1} and \ref{oplus-thm2} that if $G \oplus_v H\in G^{SSP}$ then either $G \oplus_v H=P_n \oplus_v P_m$ (with $\deg (v) = 1$ for at least one of the graphs), or
    one of the graphs is a path and the other does not admit a barbell partition.
\end{obs}

\section{Barbell Partitions and Graph Products}

We close the discussion on barbell partitions by considering such vertex partitions associated with some classical graph products. We begin by considering the corona product of two graphs.  

\begin{defn}\label{defcor}
Let $G$ and $H$ be graphs with $V(G)=\{g_i\}_{i=1}^k$ and $V(H)=\{h_j\}_{j=1}^m$.  The corona product of $G$ with $H$, denoted $G \circ H$ is the graph with vertex set $V(G \circ H)=\{g_i\}_{i=1}^k \cup \{h_{i,j}\}_{i=1,}^k{}_{j=1}^m$ and edge set $E(G \circ H)$ such that given $u,v \in V(G \circ H)$, we have $uv \in E(G \circ H)$ provided one of the following is true:
\begin{itemize}
    \item $u,v \in \{g_i\}_{i=1}^k$ and $uv \in E(G)$
    \item $u = g_i$ and $v=h_{i,j}$ for some $i \in \{1,2,...,k\}$ and some $j \in \{1,2,...,m\}$
    \item $u=h_{i,j_1}$ and $u=h_{i,j_2}$ with $h_{j_1}h_{j_2} \in E(H)$.
\end{itemize}
\end{defn}

Regarding the corona product we consider the special case of the complete graph $K_n$ with $K_1$, denoted by $K_n\circ K_1$, is called the \emph{corona} of $K_n$.  

\begin{prop}
Let 
\[B = \begin{bmatrix}
A & D_{\mu}\\
D_{\mu} & D_{\lambda}\\
\end{bmatrix} \in \mathcal{S}(K_n\circ K_1)\] 
where $A \in \mathcal{S}(K_n)$,  
$D_{\mu}=\text{diag}(\mu_1,\mu_2,\ldots,\mu_n)$ and $D_{\lambda}=\text{diag}(\lambda_1,\lambda_2,\ldots,\lambda_n)$ with 
$\mu_i \neq 0$ for all $i \in [n]$. Then $B$ has the SSP if $D_{\lambda} = \lambda I_n$, and $\mu_i \neq \pm \mu_j$ for all $i,j \in [n]$ with $i\neq j$.
\end{prop}

\begin{proof}
Since $A \in \mathcal{S}(K_n)$, the corresponding block of the matrix $X$ is the zero matrix. Let
\[X = \begin{bmatrix}
O & X_{\mu}\\
X_{\mu}^T & Y_{\lambda}\\ 
\end{bmatrix} \] 
where $[X_{\mu}]_{i,j}=x_{i,n+j}$ and $[Y_{\lambda}]_{i,j}=x_{n+i,n+j}.$ If $BX = XB,$ then

\begin{align}
 &D_{\mu}X_{\mu}^T = X_{\mu}D_{\mu}, \label{equ1corona} \\  
 &AX_{\mu}+D_{\mu}Y_{\lambda} = X_{\mu}D_{\lambda}, \label{equ2corona}\\
 &D_{\lambda}X_{\mu}^T = X_{\mu}^TA + Y_{\lambda} D_{\mu},  \label{equ3corona}\ \, \text{and}\\
 & D_{\mu}X_{\mu}+ D_{\lambda}Y_{\lambda} = X_{\mu}^TD_{\mu}+ Y_{\lambda}D_{\lambda}.
\label{equ4corona}
\end{align}

From Equation \eqref{equ1corona} we obtain: 
\begin{eqnarray}
\label{equ5corona}
[D_{\mu}X_{\mu}^T - X_{\mu}D_{\mu}]_{i,j}= \mu_i x_{j, n +i}  - \mu_j x_{i ,n+j}=0 \text{ or } x_{j n+i} = \frac{\mu_j x_{i, n+j}}{\mu_i} 
\end{eqnarray}
while from Equation \eqref{equ4corona}, we get: 
\begin{eqnarray}
\label{equ6corona}
[X_{\mu}^TD_{\mu}+ Y_{\lambda}D_{\lambda} - (D_{\mu}X_{\mu}+ D_{\lambda}Y_{\lambda})]_{i,j} = \mu_j x_{j, n + i} - \mu_i x_{i, n + j} - (\lambda_i - \lambda_j) x_{n + i, n+j} = 0. 
\end{eqnarray}
Together, we yield: 
\begin{eqnarray}
\label{equ7corona}
 \frac{(\mu_j - \mu_i) (\mu_j + \mu_i)}{\mu_i}x_{i, n+j} - (\lambda_i - \lambda_j) x_{n + i, n+j} = 0. 
\end{eqnarray}
Because $D_{\lambda} = \lambda I_n$, and $\mu_i \neq \pm \mu_j$ for all $i,j \in [n]$ with $i\neq j$, $X_{\mu}=O$. So the equations \eqref{equ1corona} to \eqref{equ4corona} simplify to 
  
  \begin{align}
& D_{\mu}Y_{\lambda} = O, \label{equ1ppcorona} \\
& O = Y_{\lambda}D_{\mu}, \, \text{and} \label{equ2ppcorona} \\
& \lambda Y_{\lambda} = \lambda Y_{\lambda} \label{equ3ppcorona}
\end{align}

Equations \eqref{equ1ppcorona}  and \eqref{equ2ppcorona} yield that $Y_{\lambda} = O$ since each $\mu_i \neq 0$ for all $i\in [n]$.
\end{proof}

Considering the corona product of graphs each with more than one vertex leads to the next result connected to barbell partitions of the corona product.

\begin{thm}
Let $G$ and $H$ be graphs each with at least two vertices. Then $G \circ H$ has a barbell partition.
\end{thm}

\begin{proof}
Let $V(G)=\{g_i\}_{i=1}^k$, $V(H)=\{h_j\}_{j=1}^m$, and $V(G \circ H)=\{g_i\}_{i=1}^k \cup \{h_{i,j}\}_{i=1,}^k{}_{j=1}^m$ as in Definition \ref{defcor}.  Since $\abs{V(G)} \geq 2$, one can let $W_1=\{h_{1,j}\}_{j=1}^m$, $W_2=\{h_{2,j}\}_{j=1}^m$, and $R=V(G) \setminus (W_1 \cup W_2)$.  We will now show that $\{R,W_1,W_2\}$ is a barbell partition of $G \circ H$.

Next note that $W_1, W_2 \neq \emptyset$ and for each $i \in \{1,2\}$, given $v \in V(G \circ H) \setminus W_i$ and $w_i \in W_i$, we have $vw_i \in E(G \circ H)$ if and only if $v=g_i$.  So there do not exist vertices $w_1 \in W_1$ and $w_2 \in W_2$ such that $w_1w_2 \in E(G \circ H)$.  In addition, since $\abs{V(H)} \geq 2$, it follows that for each $i \in \{1,2\}$, $\abs{N_{G \circ H}(g_i) \cap W_i} \geq 2$.  So for each $r \in R$ and each $i \in \{1,2\}$, $\abs{N_{G \circ H}(r) \cap W_i} \neq 1$.  Thus, $\{R, W_1, W_2\}$ is a barbell partition of $G \circ H$.
\end{proof}

If $\abs{V(G)}=1$ or $\abs{V(H)}=1$, then it is possible that $G \circ H$ admits a barbell partition.  In particular, if $\abs{V(G)}=1$ and $H$ is a graph which admits a barbell partition and has no isolated vertices, then by Corollary \ref{domin} it follows that $G \circ H$ admits a barbell partition.  In addition, if $\abs{V(H)}=1$ and $G$ is a graph which admits a barbell partition, then it follows by repeated application of Observation \ref{relpen} that $G \circ H$ admits a barbell partition.  

\begin{cor}
    Let $G$ and $H$ be graphs each with at least two vertices.  Then $G \circ H$ is not a member of  $G^{SSP}$.
\end{cor}

We next turn to the standard definitions for the Cartesian product and tensor product of two graphs.

\begin{defn}
Let $G$ and $H$ be graphs.  Then the Cartesian product of $G$ and $H$ denoted $G \Box H$ is the graph with vertex set $V(G \Box H)=V(G) \times V(H)$ and edge set $E(G \Box H)$ such that given $(g_1,h_1), (g_2,h_2) \in V(G \Box H)$, $(g_1,h_1)(g_2,h_2) \in E(G \Box H)$ provided either
\begin{itemize}
    \item $g_1=g_2$ and $h_1h_2 \in E(H)$ or
    \item $h_1=h_2$ and $g_1g_2 \in E(G)$.
\end{itemize}
\end{defn}

\begin{defn}
Let $G$ and $H$ be graphs.  Then the tensor product of $G$ and $H$ denoted $G \times H$ is the graph with vertex set $V(G \times H)=V(G) \times V(H)$ and edge set $E(G \times H)$ such that given $(g_1,h_1), (g_2,h_2) \in V(G \times H)$, $(g_1,h_1),(g_2,h_2) \in E(G \times H)$ provided $g_1g_2 \in E(G)$ and $h_1h_2 \in E(H)$.
\end{defn}

The next result represents a closure-type statement regarding the above graph products and graphs that admit barbell partitions.

\begin{thm}
Let $G$ be a graph and $H$ be a graph admitting a barbell partition, then 
    \begin{itemize}
        \item $L_1=G \Box H$ admits a barbell partition.
        \item $L_2=G \times H$ admits a barbell partition.
        %\item $K_3=G \boxtimes H$ admits a barbell partition.
    \end{itemize}
\end{thm}

\begin{proof}
%\begin{figure}[h]
  %\centering
  %\includegraphics[width=0.7\linewidth]{box_barbell_nostrong.png}
  %\caption{Illustrations of Various Graph Products}
  %\label{box_barbell}
%\end{figure}

Let $\{R, W_1, W_2\}$ be a barbell partition of $H$, and let $R'$, $W_1'$, and $W_2'$ be defined as follows:   
\begin{itemize}
    \item $R'=\left\{ (g,h) \in V(G) \times V(H): h \in R\right\}$,
    \item $W_1'=\left\{ (g,h) \in V(G) \times V(H): h \in W_1\right\}$, and
    \item $W_2'=\left\{ (g,h) \in V(G) \times V(H): h \in W_2\right\}$.
\end{itemize}
We will show that $\{R', W_1', W_2'\}$ is a barbell partition of $L_j$ for each $j \in \{1,2\}$.

\vspace{0.1in}

\begin{addmargin}[0.87cm]{0cm}{\bf Claim 1:}
$W'_1,W'_2 \not = \emptyset$ and for each $j \in \{1,2\}$, there do not exist vertices $w_1 \in W'_1$ and $w_2 \in W'_2$ such that $w_1w_2 \in E(L_j)$.
\vspace{0.1in}

\noindent {\em Proof of Claim 1.}
Since $W_1$ and $W_2$ are nonempty, $W_1'$ and $W_2'$ are nonempty.  Now, let $w_1 \in W_1'$ and $w_2 \in W'_2$ be arbitrary.  So there exist $g_1,g_2 \in V(G)$, $h_1 \in W_1$, and $h_2 \in W_2$ such that $w_1=(g_1,h_1)$ and $w_2=(g_2,h_2)$.  Since $h_1 \in W_1$ and $h_2 \in W_2$, it follows that $h_1 \neq h_2$ and $h_1h_2 \not \in E(H)$.  Thus $w_1$ and $w_2$ are not adjacent in $L_j$ for each $j \in \{1,2\}$. Finally, since $w_1$ and $w_2$ were chosen arbitrarily, for each $j \in \{1,2\}$ there do not exist vertices $w_1 \in W'_1$ and $w_2 \in W'_2$ such that $w_1w_2 \in E(L_j)$.
\end{addmargin}

\vspace{0.1in}

\begin{addmargin}[0.87cm]{0cm}{\bf Claim 2:}
For each $r \in R'$, $j \in \{1,2\}$, and $i \in \{1,2\}$, we have that $\left\vert N_{L_j}(r) \cap W'_i \right\vert \neq 1$.
\vspace{0.1in}

\noindent {\em Proof of Claim 2.}
Let $r \in R'$ and $i_0 \in \{1,2\}$ be arbitrary. Since $r \in R'$ there exist $h_r \in R$ and $g_r \in V(G)$ such that $r=(g_r,h_r)$.  Now, let $k \in \mathbb Z$ such that $\left\vert N_H(h_r) \cap W_{i_0} \right\vert =k$.  Since $\{R,W_1,W_2\}$ is a barbell partition of $H$, it follows that $k \neq 1$.

First consider $L_1$, and note that since $h_r \neq w_{i_0}$ for each $w_{i_0} \in W_{i_0}$, $r$ is neighbors with $(g,w_{i_0}) \in W'_{i_0}$ if and only if $g=g_r$ and $h_r$ is neighbors with $w_{i_0}$ in $H$.  So, it follows that
\[\left\vert N_{L_1}(r) \cap W_{i_0} \right\vert = k.\]
Next consider $L_2$, and note that for each neighbor $w_{i_0} \in W_{i_0}$ of $h_r$, $r$ has $\deg_G(g_r)$ many neighbors in $W'_{i_0}$, specifically the set of vertices 
\[\big{\{}(g,w_{i_0}): g \in N_G(g_r)\big{\}}.\]
So, it follows that
\[\left\vert N_{L_2}(r) \cap W_{i_0} \right\vert = k\cdot \deg_G(g_r).\]
%Similarly, since $r \in R'$ is neighbors with $(g,w_{i_0}) \in W'_{i_0}$ in $K_3$ if and only if $r \in R'$ is neighbors with $(g,w_{i_0}) \in W'_{i_0}$ in either $K_1$ or $K_2$, it follows that \[\left\vert N_{K_3}(r) \cap W_{i_0} \right\vert = k + k\cdot \deg_G(g_r).\]
Since $k \neq 1$, it follows that $k\cdot \deg_G(g_r) \neq 1$.  Finally, since $r \in R'$ and $i_0 \in \{1,2\}$ were arbitrary choices, for each $r \in R'$, $j \in \{1,2\}$, and $i \in \{1,2\}$, we have that $\left\vert N_{L_j}(r) \cap W'_i \right\vert \neq 1$.
\end{addmargin}
\vspace{0.1in}

\noindent Thus, $\{R', W_1', W_2'\}$ is a barbell partition of $L_j$ for each $j \in \{1,2\}$.
\end{proof}

\begin{cor}
    Let $G$ and $H$ be graphs such that $H$ admits a barbell partition.  Then 
    $G \Box H$ and $G \times H$  is not a member of  $G^{SSP}$.
\end{cor}

%\textcolor{red}{(Beginning of new material)}

\begin{thm}
Let $G$ and $H$ be graphs each containing a pair of disjoint forts.  Then $K=G \Box H$ admits a barbell partition.
\end{thm}

\begin{proof}
Let $F_G^1, F_G^2$ be a pair of forts of $G$ such that $F_G^1 \cap F_G^2=\emptyset$, and likewise let $F_H^1, F_H^2$ be a pair of forts of $H$ such that $F_H^1 \cap F_H^2=\emptyset$.  Let $W_1=\{(g,h): g \in F_G^1 \text{ and } h \in F_H^1\}$, $W_2=\{(g,h): g \in F_G^2 \text{ and } h \in F_H^2\}$, and $R=V(K) \setminus (W_1 \cup W_2)$.  We will show that $\{R,W_1,W_2\}$ is a barbell partition of $K$.

\vspace{0.1in}

\begin{addmargin}[0.87cm]{0cm}{\bf Claim 1:}
$W_1,W_2 \not = \emptyset$ and for each $j \in \{1,2\}$, there do not exist vertices $w_1 \in W_1$ and $w_2 \in W_2$ such that $w_1w_2 \in E(K)$.

\vspace{0.1in}

\noindent {\em Proof of Claim 1.}
By construction $W_1,W_2 \neq \emptyset$.  Furthermore, given $(g_1,h_1) \in W_1$ and $(g_2,h_2) \in W_2$, since $F_G^1 \cap F_G^2= \emptyset$ and $F_H^1 \cap F_H^2 =\emptyset$, it follows that $g_1 \neq g_2$ and $h_1 \neq h_2$.  Thus $(g_1,h_1)(g_2,h_2) \not \in E(K)$.  Furthermore, since $(g_1,h_1)$ and $(g_2,h_2)$ were chosen arbitrarily, there do not exist vertices $w_1 \in W_1$ and $w_2 \in W_2$ such that $w_1w_2 \in E(K)$.
\end{addmargin}

\vspace{0.1in}

\begin{addmargin}[0.87cm]{0cm}{\bf Claim 2:}
For each $r \in R$ and $i \in \{1,2\}$, we have that $\left\vert N_{K}(r) \cap W_i \right\vert \neq 1$.
\vspace{0.1in}

\noindent {\em Proof of Claim 2.}
First consider $W_1$.  For each $r=(g,h) \in R$, either $g \not \in F_G^1$ or $h \not \in F_H^1$.  First if $g \not \in F_G^1$ and $h \not \in F_H^1$, then $|N_K(r) \cap W_1|=0$.  So without loss of generality suppose $g \not \in F_G^1$ and $h \in F_H^1$.  Since $g \not \in F_G^1$, there exists $k \in \mathbb N$ such that $\abs{N_G(g) \cap F_G^1}=k \neq 1$.  Furthermore for each $g' \in F_G^1$, $r=(g,h)$ is neighbors with $(g',h) \in W_1$ if and only if $g$ is neighbors with $g'$ in $G$.  So, it follows that $\left\vert N_{K}(r) \cap W_1 \right\vert = k$.  Thus, for each $r=(g,h) \in R$, $\left\vert N_{K}(r) \cap W_1 \right\vert \neq 1$.  A similar argument shows that for each $r=(g,h) \in R$, $\left\vert N_{K}(r) \cap W_2 \right\vert \neq 1$.
\end{addmargin}

\vspace{0.1in}

Thus $\{R,W_1,W_2\}$ is a barbell partition of $K=G\Box H$.
\end{proof}

Many graphs contain a pair of disjoint forts. Examples of graphs which contain a disjoint pair of forts, but which do not admit barbell partitions are complete graphs $K_n$, $n \geq 4$ and even cycles.

\begin{cor}
Let $m,k \in \mathbb Z^+$ with $m,k \geq 4$.  Then $G=K_m \Box K_k$ admits a barbell partition. 
\end{cor}

\begin{cor}
Let $C_m$ and $C_k$ each be even cycles with $m,k\geq 4$. Then $G=C_m \Box C_k$ admits a barbell partition.
\end{cor}

Path graphs and odd cycles are examples of graphs which do not contain a pair of disjoint forts.  It is easy to check that $P_2 \Box P_3$ does not admit a barbell partition.  However, while grid graphs do not seem to admit barbell partitions, the following theorem shows that certain prism grid graphs do admit barbell partitions, even when the cycles they are built from may be of odd length.

%\textcolor{red}{(End of new material). This looks good to me - thanks Houston (SF)}

\begin{thm}
Let $m,k \in \mathbb Z^+$ such that $k \geq 4$, then $K=C_k \Box C_{mk}$ admits a barbell partition.
\end{thm}

\begin{figure}[h]
  \centering
  \includegraphics[width=0.4\linewidth]{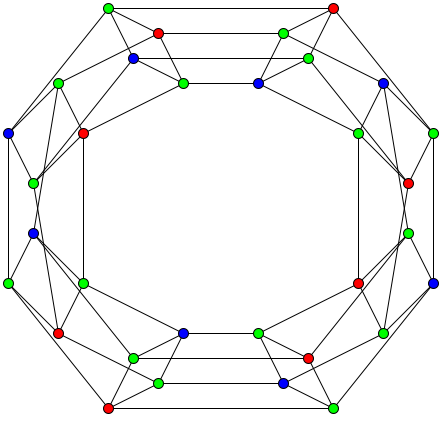}
  \caption{$C_4 \Box C_{8}$}
  \label{join-2}
\end{figure}

\setcounter{case}{0}

\begin{proof}

For $N \in \{k,mk\}$, enumerate the vertices of $C_N$ as $\{j\}_{j=1}^N$ such that given $j_1,j_2 \in V(C_N)$, we have that $j_1j_2 \in E(C_N)$ provided either $\vert j_1-j_2 \vert=1$ or $\{j_1,j_2\}=\{1,N\}$, and let $\{R,W_1,W_2\}$ be such that
\begin{itemize}
    \item $W_1=\left\{(j_1,j_2) \in V(C_k) \times V(C_{mk}): j_2=Mk+j_1 \text{ with }0\leq M \leq m-1\right\}$,
    \item $R=\left\{(j_1,j_2) \in V(C_k) \times V(C_{mk}): j_2=Mk+j_1+1 \text{ or } j_2=Mk+j_1+(k-1) \text{ with } -1\leq M \leq m-1\right\}$,
    \item $W_2=\left\{(j_1,j_2) \in V(C_k) \times V(C_{mk}): j_2=Mk+j_1+N \text{ with }-1\leq M \leq m-1 \text{ and } 2 \leq N \leq k-2\right\}$.
\end{itemize}

An example is shown in Figure \ref{join-2} where $R$ is formed by the green vertices, $W_1$ is the set of red vertices, and  $W_2$ is the set of blue vertices. We will now show that $\{R,W_1,W_2\}$ is a barbell partition of $G$. 
%\textcolor{blue}{Alternate version of the argument presented here in red so that no one ever has to read this mess again (I am quite happy with my result but the calculations and assessing them are a nuisance).-Houston}\textcolor{red}{By the nature of the construction it follows that 
%\begin{itemize}
    %\item $W_1,W_2,R \neq \emptyset$,
    %\item if $w \in W_1$ and $wv \in E(K)$, then $v \in R$,
    %\item if $r \in R$ and $rv \in E(K)$, then $v \not \in R$, and
    %\item if $r \in R$, then $\abs{N_K(r) \cap W_1}=2$.
%\end{itemize}  
%Since $K$ is $4$-regular, it follows that $\{R,W_1,W_2\}$ is a barbell partition of $K$.}

\vspace{0.1in}

\begin{addmargin}[0.87cm]{0cm}{\bf Claim 1:}
$W_1,W_2 \not = \emptyset$ and there do not exist vertices $w_1 \in W_1$ and $w_2 \in W_2$ such that $w_1w_2 \in E(K)$.
\vspace{0.1in}

\noindent {\em Proof of Claim 1.}
Since $(1,1) \in W_1$ and $(1,c) \in W_2$, it is clear that $W_1,W_2 \neq \emptyset$.  Next, note that $\{R,W_1,W_2\}$ is a partition of $V(G)$ and let $v \in W_1$ be arbitrary.  We will show that $N_K(v) \subset R$ and thus, since $R \cap W_2 = \emptyset$, that there do not exist vertices $w_1 \in W_1$ and $w_2 \in W_2$ such that $w_1w_2 \in E(K)$.  Since $v \in W_1$ it follows that $v$ is of the form $(j_1, j_2)$ with $j_2=Mk+j_1$ and $0\leq M \leq m-1$.  To show that $N_K(v) \subset R$, we will show that there exist distinct values $a,b$ such that $(j_1,a),(j_1,b) \in N_K(v) \cap R$ and distinct values $c,d$ such that $(c,j_2),(d,j_2) \in N_K(v) \cap R$, and thus since $K$ is 4-regular, that $N_K(v) \subset R$.

\begin{case}
$j_2 \not \in \{1,mk\}$
\end{case}

Since $j_2 \not \in \{1,mk\}$, it follows that $(j_1,j_2)(j_1,j_2+1),(j_1,j_2)(j_1,j_2-1) \in E(K)$.  It now remains to argue that $(j_1,j_2-1),(j_1,j_2+1) \in R$.  First, since $1< j_2 <mk$ and $1 \leq j_1 \leq k$, we have that $1 \leq j_2-1 < j_2+1 \leq mk$ and $1 \leq j_1 \leq k$.  Since $j_2=Mk+j_1$ for some $M$ with $0 \leq M \leq m-1$, $j_2+1=Mk+j_1+1$ and so $(j_1,j_2+1) \in R$.  In addition, since $M \geq 0$, $j_2-1=Mk+j_1-1=(M-1)k+j_1+(k-1)$ with $M-1 \geq -1$ and so $(j_1,j_2-1) \in R$.

\begin{case}
$j_2=1$
\end{case}

Since $j_2=1$, it follows that $(j_1,j_2)(j_1,2),(j_1,j_2)(j_1,mk) \in E(K)$.  It now remains to argue that $(j_1,2),(j_1,mk) \in R$.  First note, $1 \leq j_1 \leq k$.  Since $j_2=Mk+j_1=1$, we have that $M=0$ and $j_1=1$, so $2=Mk+j_1+1$ and thus $(j_1,2) \in R$.  In addition, $mk=(m-1)k+j_1+(k-1)$ and thus $(j_1,mk) \in R$.

\begin{case}
$j_2=mk$
\end{case}

This case follows from similar lines of reasoning as in Case 2.
%Since $j_2=mk$, it follows that $(c_{j_1},c_{j_2})(c_{j_1},c_{mk-1}),(c_{j_1},c_{j_2})(c_{j_1},c_{1}) \in E(K)$.  It now remains to argue that $(c_{j_1},c_{mk-1}),(c_{j_1},c_{1}) \in R$.  First note, $1 \leq j_1 \leq k$.  Since $j_2=Mk+j_1=mk$, we have that $M=m-1$ and $j_1=k$, so $mk-1=(m-2)k+j_1+(k-1)$.  Since $m \geq 1$ we have $m-2 \geq -1$, and thus $(c_{j_1},c_{mk-1}) \in R$.  In addition, $1=-k + k +1$ and thus $(c_{j_1},c_{1}) \in R$.

\vspace{0.1in}

\noindent In any case, there exist distinct values $a,b$ such that $(j_1,a),(j_1,b) \in N_K(v) \cap R$.  It now remains to show that there exist distinct values $c,d$ such that $(c,j_2),(d,j_2) \in N_K(v) \cap R$.  Before considering the cases, note that if $x=j_2-Mk-1$ or $x=j_2-Mk-(k-1)$ with $-1 \leq M \leq m-1$, then $(x,j_2) \in R$.

\setcounter{case}{0}

\begin{case}
$j_1 \not \in \{1,k\}$
\end{case}

Since $j_1 \not \in \{1,k\}$, it follows that $(j_1,j_2)(j_1+1,j_2),(j_1,j_2)(j_1-1,j_2) \in E(K)$.  It now remains to argue that $(j_1+1,j_2),(j_1-1,j_2) \in R$.  First, since $1< j_1 <k$ and $1 \leq j_2 \leq mk$, we have that $1 \leq j_1-1 < j_1+1 \leq k$ and $1 \leq j_2 \leq mk$.  Since $j_2=j_1+Mk$ for some $M$ with $0 \leq M \leq m-1$, $j_1-1=j_2-Mk-1$ and so $(j_1-1,j_2) \in R$.  In addition, $j_1+1=j_2-(M-1)k-(k-1)$ and since $M \geq 0$ we have $M-1 \geq -1$, and thus $(j_1+1,j_2) \in R$.  

\begin{case}
$j_1=1$
\end{case}

Since $j_1=1$, it follows that $(j_1,j_2)(2,j_2),(j_1,j_2)(k,j_2) \in E(K)$.  It now remains to argue that $(2,j_2),(k,j_2) \in R$.  First note, $1 \leq j_2 \leq mk$.  Since $j_2=j_1+Mk$ we have $j_2=Mk+1$ for some $M$ with $0 \leq M \leq m-1$.  So $2=j_2-((M-1)k)-(k-1)$ and $k=j_2-((M-1)k)-1$.  Since $M \geq 0$, we have $M-1 \geq -1$, and thus $(2,j_2),(k,j_2) \in R$.

\begin{case}
$j_1=k$
\end{case}

This case follows from similar lines of reasoning as in Case 2
%Since $j_1=k$, it follows that $(c_{j_1},c_{j_2})(c_{k-1},c_{j_2}),(c_{j_1},c_{j_2})(c_{1},c_{j_2}) \in E(K)$.  It now remains to argue that $(c_{k-1},c_{j_2}),(c_{1},c_{j_2}) \in R$.  First note, $1 \leq j_2 \leq mk$.  Since $j_2=j_1+Mk$ we have $j_2=(M+1)k$ for some $M$ with $0 \leq M \leq m-1$.  So $k-1=j_2-Mk-1$ and $1=j_2-Mk-(k-1)$, and thus $(c_{k-1},c_{j_2}),(c_{1},c_{j_2}) \in R$.

\end{addmargin}

\vspace{0.1in}

\begin{addmargin}[0.87cm]{0cm}{\bf Claim 2:}
For each $r \in R$ and $i \in \{1,2\}$, we have that $\left\vert N_K(r) \cap W_i \right\vert =2$.
\vspace{0.1in}

\noindent {\em Proof of Claim 2.}
Using techniques similar to those exhibited in Claim 1, the reader can check that for each $r \in R$ there exist values $a,b,c,d$ such that such that $(j_1,a),(b,j_2) \in N_K(r) \cap W_1$ and distinct values $c,d$ such that $(j_1,c),(d,j_2) \in N_K(r) \cap W_2$, with $r$ of the form $(j_1,j_2)$.  Thus since $K$ is 4-regular it follows that for each $r \in R$ and each $i \in \{1,2\}$ we have $\left\vert N_K(r) \cap W_i \right\vert = 2$.
\end{addmargin}
\vspace{0.1in}

\noindent Thus, $\{R, W_1, W_2\}$ is a barbell partition of $K$.
\end{proof}

\begin{thm} \label{thmtensnopen}
Let $G$ be a graph which is not a complete graph and $H$ be a graph with no pendant vertices such that $\abs{V(H)} \geq 2$.  Then $K=G \times H$ admits a barbell partition.
\end{thm}

%\begin{figure}[h]
  %\centering
  %\includegraphics[width=0.25\linewidth]{timesprism2.png}
%\end{figure}

\begin{proof}
If $G$ or $H$ are disconnected, then $K=G \times H$ is disconnected in which case by Observation \ref{discon}, $K$ admits a barbell partition.  So suppose both $G$ and $H$ are connected.  Since $H$ is connected and has no pendant vertices, it follows that $\delta(H) \geq 2$.  Since $G$ is not a complete graph there exist vertices $u,v \in V(G)$ such that $uv \not \in E(G)$.  Let
\[W_1=\left\{(u,h): h \in V(H)\right\}, W_2=\left\{(v,h): h \in V(H)\right\}, \text{ and } R=\left\{(g,h): g \in V(G) \setminus \{u,v\} \text{ and } h \in V(H)\right\}.\]
We will show that $\{R,W_1,W_2\}$ is a barbell partition of $K$.

\begin{addmargin}[0.87cm]{0cm}{\bf Claim 1:}
$W_1,W_2 \not = \emptyset$ and there do not exist vertices $w_1 \in W_1$ and $w_2 \in W_2$ such that $w_1w_2 \in E(K)$.

\vspace{0.1in}

\noindent {\em Proof of Claim 1.}
Since $V(H) \neq \emptyset$, it follows that $W_1,W_2 \not = \emptyset$.  Furthermore, since $uv \not \in E(G)$, there do not exist vertices $w_1 \in W_1$ and $w_2 \in W_2$ such that $w_1w_2 \in E(K)$.
\end{addmargin}

\vspace{0.1in}

\begin{addmargin}[0.87cm]{0cm}{\bf Claim 2:}
For each $r \in R$, $\left\vert N_G(r) \cap W_i \right\vert \neq 1$ for $i \in \{1,2\}$.
\vspace{0.1in}

\noindent {\em Proof of Claim 2.}
Let $r \in R$ be arbitrary.  Since $r \in R$, there exists $g \in V(G) \setminus \{u,v\}$ and $h \in V(H)$ such that $r=(g,h)$.  First, note that if $gu \not \in E(G)$, then $r$ will have no neighbors in $W_1$.  So suppose $gu \in E(G)$.  Since $H$ is such that $\delta(H)\geq 2$ there exists $h_1,h_2 \in V(H)$ such that $hh_1,hh_2 \in E(H)$.  Since $gu \in E(G)$ and $hh_1,hh_2 \in E(H)$, $r$ will be adjacent to both $uh_1$ and $uh_2$, each of which are members of $W_1$.  So $\left\vert N_K(r) \cap W_1 \right\vert \geq 2$.  In either case, for each $r \in R$ we have that $\left\vert N_K(r) \cap W_1 \right\vert \neq 1$.  Likewise, for each $r \in R$, $\left\vert N_K(r) \cap W_2 \right\vert \neq 1$.
\end{addmargin}

\vspace{0.1in}

Thus $\{R, W_1, W_2\}$ forms a barbell partition of $K=G \times H$.
\end{proof}

\begin{ex}
The requirement in Theorem \ref{thmtensnopen} that $H$ has no pendant vertices may seem unnecessary. However, we observe, by example, a pair of graphs $G$ and $H$ where $G$ is not complete, $H$ contains at least two vertices (each pendant vertices), but the product $G \times H$ does not admit a barbell partition.  This identifies that the criterion regarding pendant vertices is quite necessary.

\begin{figure}[h!]
	\begin{subfigure}[b]{.3\textwidth}
			\begin{center}
			\begin{tikzpicture}[scale=2,thick]
				\tikzstyle{every node}=[minimum width=0pt, inner sep=2pt, circle]
				\draw (0.16,0.02) node[draw] (1) { \tiny 1};
				\draw (0.16,1.1) node[draw] (2) { \tiny 2};
				\draw (-0.625,1.9) node[draw] (3) { \tiny 3};
				\draw (-1.41,1.1) node[draw] (4) { \tiny 4};
				\draw (-1.41,0.02) node[draw] (5) { \tiny 5};
				\draw  (1) edge (2);
				\draw  (2) edge (3);
				\draw  (4) edge (5);
				\draw  (1) edge (5);
				\draw  (3) edge (4);
				\draw  (2) edge (4);
			\end{tikzpicture}
		\caption*{$H$}
		\end{center}
	\end{subfigure}
	\begin{subfigure}[b]{.6\textwidth}
		\begin{center}
			\begin{tikzpicture}[scale=2,thick]
				\tikzstyle{every node}=[minimum width=0pt, inner sep=2pt, circle]
				\draw (-4,1) node[draw] (0) { \tiny 5,a};
				\draw (-4,0) node[draw] (1) { \tiny 1,b};
				\draw (-3,1) node[draw] (2) { \tiny 4,b};
				\draw (-3,0) node[draw] (3) { \tiny 2,a};
				\draw (-2,1) node[draw] (4) { \tiny 3,a};
				\draw (-2,0) node[draw] (5) { \tiny 3,b};
				\draw (-1,1) node[draw] (6) { \tiny 2,b};
				\draw (-1,0) node[draw] (7) { \tiny 4,a};
				\draw (0,0) node[draw] (8) { \tiny 5,b};
				\draw (0,1) node[draw] (9) { \tiny 1,a};
				\draw  (0) edge (2);
				\draw  (2) edge (3);
				\draw  (0) edge (1);
				\draw  (1) edge (3);
				\draw  (2) edge (4);
				\draw  (3) edge (5);
				\draw  (4) edge (6);
				\draw  (5) edge (7);
				\draw  (6) edge (7);
				\draw  (6) edge (9);
				\draw  (8) edge (9);
				\draw  (7) edge (8);
			\end{tikzpicture}
		\caption*{$K_2\times H$}
		\end{center}
	\end{subfigure}
	\caption{}
\end{figure}

Assume $H\times K_2$ has a barbell partition. Now, assume vertex $(3,a)$ is in $R$. Then, both $(4,b)$ and $(2,b)$ need to be either in $R$ or (without loss of generality) in $W_1$. If they are both in $R$. Then note that both $(5,a)$ and $(2,a)$ are either in $R$ or in $W_1$. In the first case, that would imply that $(1,b)$ and $(3,b)$ are in $R$ as well as $(4,a)$, $(5,b)$ and finally $(1,a)$, that is, the whole graph is in $R$, which is a contradiction. On the other hand if both $(5,a)$ and $(2,a)$ are in $W_1$, that would imply that $(3,b)$ is not in $W_2$ and that $(4,a)$ is either in $R$ or $W_1$ (not in $W_2$), then $(5,b)$ is not in $W_2$ and finally $(1,a)$ would not be in $W_2$ either and $W_2 =\emptyset$, a contradiction. Now, let us assume that both $(4,b)$ and $(2,b)$ are in $W_1$. This implies that the only vertex that can be in $W_2$ is $(3,b)$, a contradiction. Assuming that $(3,a)$ is in $W_1$ we also find a contradiction. Thus $H\times K_2$ has no barbell partition.
\end{ex}

Before we close the topic of tensor products we consider the tensor product of two path graphs and the tensor product of two complete graphs.  First we have the following. 

\begin{thm} \label{thmpathstens}
Let $m,k \in \mathbb Z^+$ such that $m,k \geq 2$, then $G=P_k \times P_m$ admits a barbell partition.
\end{thm}

\begin{proof}
$G$ is disconnected, so it immediately admits a barbell partition.
\end{proof}

On the other hand, results concerning complete graphs are a bit more nuanced. It can be checked by exhaustion that $K_3 \times K_3$ does not admit a barbell partition.  However, as witnessed by the following theorem, there are instances in which $K_n \times K_m$ does admit a barbell partition.

\begin{thm}
Let $n,m \in \mathbb N$ with $n \geq 2$ and $m \geq 6$.  Then $G=K_n \times K_m$ admits a barbell partition.
\end{thm}

\begin{figure}[h]
  \centering
  \includegraphics[width=0.6\linewidth]{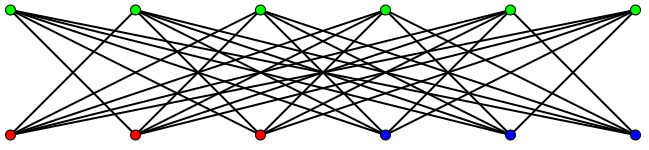}
  \label{join-3}
\end{figure}

\begin{proof}
Enumerate the vertices of $K_n = \{u_i\}_{i=1}^n$ and the vertices of $K_m = \{v_j\}_{j=1}^m$.  Let 
\[R=\left\{(u_i,v): i> 1 \text{ and }v \in V(K_m)\right\},\]
\[W_1=\left\{(u_1,v_j): j \leq \left\lceil \frac{m}{2} \right\rceil \right\}, \text{ and } W_2=\left\{(u_1,v_j): j > \left\lceil \frac{m}{2}\right\rceil \right\}.\]
We will show that $\{R,W_1,W_2\}$ is a barbell partition of $G$.

\begin{addmargin}[0.87cm]{0cm}{\bf Claim 1:}
$W_1,W_2 \not = \emptyset$ and there do not exist vertices $w_1 \in W_1$ and $w_2 \in W_2$ such that $w_1w_2 \in E(G)$.

\vspace{0.1in}

\noindent {\em Proof of Claim 1.}
Since $m > 2$, $W_1,W_2 \not = \emptyset$.  Furthermore, since $W_1 \cup W_2=\left\{(u_1,v): v \in V(K_m)\right\}$, it follows that $W_1 \cup W_2$ is an independent set of vertices in $G=K_n \times K_m$.
\end{addmargin}

\vspace{0.1in}

\begin{addmargin}[0.87cm]{0cm}{\bf Claim 2:}
For each $r \in R$, $\left\vert N_G(r) \cap W_i \right\vert \neq 1$ for $i \in \{1,2\}$.
\vspace{0.1in}

\noindent {\em Proof of Claim 2.}
Given $r \in R$, $r$ is of the form $(u_i,v_j)$ with $i>1$.  So $r$ is adjacent to every member of $W_1 \cup W_2$ except for $(u_1,v_j)$, and since $m \geq 6$, it follows that 
\[\left\vert N_G(r) \cap W_i \right\vert \geq \left\lfloor \frac{m}{2} \right\rfloor -1 \geq 2, \text{ for each } i \in \{1,2\}.\]
\end{addmargin}

\vspace{0.1in}

Thus $\{R, W_1, W_2\}$ forms a barbell partition of $G=K_n \times K_m$.
\end{proof}

We finish this section by considering one final graph product known as the strong product.

\begin{defn}
Let $G$ and $H$ be graphs.  Then the strong product of $G$ and $H$ denoted by $G \boxtimes H$ is the graph with vertex set $V(G \boxtimes H)=V(G) \times V(H)$ and edge set $E(G \boxtimes H)$ such that given $(g_1,h_1), (g_2,h_2) \in V(G \boxtimes H)$, $(g_1,h_1)(g_2,h_2) \in E(G \boxtimes H)$ provided either
\begin{itemize}
    \item $g_1=g_2$ and $h_1h_2 \in E(H)$ or
    \item $h_1=h_2$ and $g_1g_2 \in E(G)$ or
    \item $g_1g_2 \in E(G)$ and $h_1h_2 \in E(H)$.
\end{itemize}
\end{defn}

\begin{obs}
For $m,n\geq 1$ the graph $K_n \boxtimes K_m$ is isomorphic to $K_{n+m}$, and so does not admit a barbell partition.
\end{obs}

\begin{thm}
Let $G$ be a graph which is not a complete graph and $H$ be a graph with more than one vertex.  Then $K=G \boxtimes H$ admits a barbell partition.
\end{thm}

%\begin{figure}[h]
  %\centering
  %\includegraphics[width=0.45\linewidth]{strongprod.png}
  %\label{join-1}
%\end{figure}

\begin{proof}
If $G$ or $H$ are disconnected, then $K=G \boxtimes H$ is disconnected in which case by Observation \ref{discon}, $K$ admits a barbell partition.  So suppose both $G$ and $H$ are connected.  Since $G$ is not a complete graph there exist vertices $u,v \in V(G)$ such that $uv \not \in E(G)$.  Let
\[W_1=\left\{(u,h): h \in V(H)\right\}, W_2=\left\{(v,h): h \in V(H)\right\}, \text{ and } R=\left\{(g,h): g \in V(G) \setminus \{u,v\} \text{ and } h \in V(H)\right\}.\]
We will show that $\{R,W_1,W_2\}$ is a barbell partition of $K$.

\begin{addmargin}[0.87cm]{0cm}{\bf Claim 1:}
$W_1,W_2 \not = \emptyset$ and there do not exist vertices $w_1 \in W_1$ and $w_2 \in W_2$ such that $w_1w_2 \in E(K)$.

\vspace{0.1in}

\noindent {\em Proof of Claim 1.}
Since $V(H) \neq \emptyset$, it follows that $W_1,W_2 \not = \emptyset$.  Furthermore, since $u \neq v$ and $uv \not \in E(G)$, there do not exist vertices $w_1 \in W_1$ and $w_2 \in W_2$ such that $w_1w_2 \in E(K)$.
\end{addmargin}

\vspace{0.1in}

\begin{addmargin}[0.87cm]{0cm}{\bf Claim 2:}
For each $r \in R$, $\left\vert N_G(r) \cap W_i \right\vert \neq 1$ for $i \in \{1,2\}$.
\vspace{0.1in}

\noindent {\em Proof of Claim 2.}
Let $r \in R$ be arbitrary.  Since $r \in R$, there exists $g \in V(G) \setminus \{u,v\}$ and $h \in V(H)$ such that $r=(g,h)$.  First, note that if $gu \not \in E(G)$, then $r$ will have no neighbors in $W_1$.  So suppose $gu \in E(G)$.  Since $H$ is a connected graph on at least two vertices there exists $h' \in V(H)$ such that $hh' \in E(H)$.  Since $gu \in E(G)$ and $hh' \in E(H)$, $r$ will be adjacent to both $uh$ and $uh'$, each of which are members of $W_1$.  So $\left\vert N_K(r) \cap W_1 \right\vert \geq 2$.  In either case, for each $r \in R$ we have that $\left\vert N_K(r) \cap W_1 \right\vert \neq 1$.  Likewise, for each $r \in R$, $\left\vert N_K(r) \cap W_2 \right\vert \neq 1$.
\end{addmargin}

\vspace{0.1in}

Thus $\{R, W_1, W_2\}$ forms a barbell partition of $K=G \boxtimes H$.
\end{proof}

\begin{cor}
    Let $G$ be a graph which is not a complete graph and $H$ be a graph of more than one vertex.  Then $K=G \boxtimes H$ is not a member of $G^{SSP}$.
\end{cor}

\section*{Acknowledgements}

This work started at the MRC workshop “Finding Needles in Haystacks: Approaches to Inverse
Problems using Combinatorics and Linear Algebra”, which took place in June 2021 with support from the
National Science Foundation and the American Mathematical Society. The authors are grateful to the
organizers of this meeting. In particular, this material is based upon work supported by the National Science Foundation under Grant Number DMS 1916439%1641020.

Shaun M. Fallat was also supported in part by an NSERC Discovery Research Grant, Application No.: RGPIN--2019--03934.
%%%%%%%%%%%%%%%%%%%%%%%%%%%%%%%%%%%%%%%%%%%%%%%%%%%%%%%%%%%%%%%%%%%%%%%%%


\begin{thebibliography}{100}



\bibitem{AACF} B. Ahmadi, F. Alinaghipour, M.S. Cavers, S. M. Fallat, K. Meagher, and S. Nasserasr, Minimum number of distinct eigenvalues of graphs, \emph{Electron. J. Linear Algebra} 26 (2013)   673--691.

\bibitem{AAB} J. Ahn, C. Alar, B. Bjorkman, S. Butler, J. Carlson, A. Goodnight, H. Knox, C. Monroe, M.C. Wigal, Ordered multiplicity inverse eigenvalue problem for graphs on six vertices,  \emph{Electron. J. Linear Algebra} 37 (2021), 316–358.


\bibitem{Arnold71} V. I. Arnol'd.  On matrices depending on parameters. (Russian) \emph{Uspehi Mat.\ Nauk,} 26:101--114, 1971.  (English translation in 
\emph{Russian Math.\ Surveys}, 26:29--43, 1971.)


\bibitem{BFH} W. Barrett, S. Fallat, H.T. Hall, L. Hogben, J.C.-H. Lin, B.L. Shader, Generalizations of the Strong Arnold Property and the minimum number of distinct eigenvalues of a graph, \emph{Electron. J. Combin} 24(2) (2017) 1--28.


\bibitem{IEPG2}
W.~Barrett, S.~Butler, S.~M. Fallat, H.~T. Hall, L.~Hogben, J.~C.-H. Lin,
  B.~L.~Shader, and M.~Young.
\newblock The inverse eigenvalue problem of a graph: {M}ultiplicities and
  minors.
\newblock {\em J. Combin. Theory Ser.~B,}, 142:276--306, 2020.


 \bibitem{FHLS}
 S. Fallat, H.T. Hall, J. C.-H. Lin, B. Shader. The bifurcation lemma for strong properties in the inverse eigenvalue problem of a graph. \emph{Linear Algebra Appl.} 648 (2022), 70–87. 

% \bibitem{cv} D. Cvetkovi\'c, P. Rowlinson,  S. Simi\'c, {\it An introduction to the theory of graph spectra\/}, Cambridge University Press, Cambridge, 2012.

 \bibitem{FH} S. Fallat, L. Hogben, The minimum rank of symmetric matrices described by a graph: A survey, \emph{Linear Algebra Appl.} 426 (2007) 558–582.


 \bibitem{Fal} S. Fallat, L. Hogben, Variants on the minimum rank problem: A survey II, Preprint, arXiv:1102-5142v1, 2011.
 
 

\bibitem{HLA2}	S. Fallat and L. Hogben. Minimum Rank, Maximum Nullity, and Zero Forcing Number of Graphs.  In \emph{Handbook of Linear Algebra},  2nd edition, L. Hogben editor, pp. 46-1\,--\,46-36, CRC Press, Boca Raton, 2014. 

 %\bibitem{GR} C. Godsil, G. Royle, {\it Algebraic Graph Theory\/}, Springer, 2001.

\bibitem{fort}
C.C. Fast and I.V. Hicks, Effects of vertex degrees on the zero-forcing number and
propagation time of a graph, \emph{Discr. Appl. Math.} 250 (2018) 215–226.

%\bibitem{Hog} L. Hogben, Spectral graph theory and the inverse eigenvalue problem of a graph, Electron. J. Linear Algebra, 14 (2005) 12--31.

%\bibitem{Hor} R. Horn, C. Johnson, Matrix Analysis, Cambridge Univ. Press, Cambridge, 1989.

\bibitem{HLS2022inverse} Hogben, L. and Lin, J.C.H. and Shader, B.L., \emph{Inverse Problems and Zero Forcing for Graphs,}
Mathematical Surveys and Monographs, https://books.google.com/books?id=cKF9EAAAQBAJ, American Mathematical Society, 2022.


\bibitem{LOS} R. H. Levene, P. Oblak, H. \v{S}migoc, A Nordhaus–Gaddum conjecture for the minimum number of distinct eigenvalues of a graph, \emph{Linear Algebra
Appl.} 564 (2019) 236--263.

\bibitem{SSPforGs} Lin, J. C. H., Oblak, P., and Šmigoc, H. (2020). The strong spectral property for graphs. {\em Linear Algebra Appl.}, 598, 68-91.


\bibitem{CdVF}
Y.~Colin de Verdi\`ere.
\newblock Sur un nouvel invariant des graphes et un crit\`ere de planarit\'e.
\newblock {\em J. Combin. Theory Ser.~B,}, 50:11--21, 1990.

\bibitem{CdV}
Y.~Colin de Verdi\`ere.
\newblock On a new graph invariant and a criterion for planarity.
\newblock In {\it Graph Structure Theory}, pp. 137--147, American Mathematical Society, Providence, RI, 1993.
 


\end{thebibliography}
\end{document}